\documentclass[11pt]{article}
\usepackage{amsmath,graphicx,amsthm,amssymb,algorithm,algorithmic}
\usepackage[%paperwidth=18.4cm, paperheight= 26cm,
            body={15.2true cm,24true cm},
            headheight=-1.0true cm
            ]{geometry}
\newtheorem{theorem}{Theorem}
\newtheorem{lemma}{Lemma}

\title{On Convergence of the Inexact Rayleigh Quotient Iteration
with MINRES\footnote{Supported by the National Science Foundation of
China (No. 10771116) and the Doctoral Program of the Ministry of
Education (No. 20060003003).}}

\author{Zhongxiao Jia\thanks{Department of Mathematical
Sciences, Tsinghua University, Beijing 100084, People's Republic
of China, jiazx@tsinghua.edu.cn}}
\date{}

\begin{document}
\maketitle

%%%%%%%%%%%%%%%%%%%%%%%%%%%%%%%%%%%%%%%%%%%%%%%%%%%%%%%%%
%%%%%  Abstract and keywords
%%%%%%%%%%%%%%%%%%%%%%%%%%%%%%%%%%%%%%%%%%%%%%%%%%%%%%%%%

\begin{abstract}
For the Hermitian inexact Rayleigh quotient iteration (RQI), we
present a new general theory, independent of iterative solvers for
shifted inner linear systems. The theory shows that the method
converges at least quadratically under a new condition, called the
uniform positiveness condition, that may allow inner tolerance
$\xi_k\geq 1$ at outer iteration $k$ and can be considerably weaker
than the condition $\xi_k\leq\xi<1$ with $\xi$ a constant not near
one commonly used in literature. We consider the convergence of the
inexact RQI with the unpreconditioned and tuned preconditioned
MINRES method for the linear systems. Some attractive properties are
derived for the residuals obtained by MINRES. Based on them and the
new general theory, we make a more refined analysis and establish a
number of new convergence results. Let $\|r_k\|$ be the residual
norm of approximating eigenpair at outer iteration $k$. Then all the
available cubic and quadratic convergence results require
$\xi_k=O(\|r_k\|)$ and $\xi_k\leq\xi$ with a fixed $\xi$ not near
one, respectively. Fundamentally different from these, we prove that
the inexact RQI with MINRES generally converges cubically,
quadratically and linearly provided that $\xi_k\leq\xi$ with a
constant $\xi<1$ not near one, $\xi_k=1-O(\|r_k\|)$ and
$\xi_k=1-O(\|r_k\|^2)$, respectively. Therefore, the new convergence
conditions are much more relaxed than ever before.
The theory can be used to design practical stopping criteria to
implement the method more effectively.
Numerical experiments confirm our results.
\bigskip


\bigskip


\end{abstract}

%%%%%%%%%%%%%%%%%%%%%%%%%%%%%%%%%%%%%%%%%%%%%%%%%%%%%%%%%
%%%%%  Introduction
%%%%%%%%%%%%%%%%%%%%%%%%%%%%%%%%%%%%%%%%%%%%%%%%%%%%%%%%%

\section{Introduction} \label{SecIntro}

We consider the problem of computing an eigenvalue $\lambda$ and the
associated eigenvector $x$ of a large and possibly sparse Hermitian
matrix $A\in \mathbb{C}^{n\times n}$, i.e.,
\begin{equation}
Ax=\lambda x.
\end{equation}
Throughout the paper, we are interested in the eigenvalue
$\lambda_1$ closest to a target $\sigma$ and its corresponding
eigenvector $x_1$ in the sense that
\begin{equation}
|\lambda_1-\sigma| <|\lambda_2-\sigma| \le \cdots \le
|\lambda_n-\sigma|. \label{sigma}
\end{equation}
Suppose that $\sigma$ is between $\lambda_1$ and $\lambda_2$. Then we have
\begin{equation}
|\lambda_1-\sigma|<\frac{1}{2}|\lambda_1-\lambda_2|.\label{gap}
\end{equation}
We denote by $x_1, x_2,\ldots, x_n$ are the unit length eigenvectors
associated with $\lambda_1,\lambda_2,\ldots,\lambda_n$. For brevity
we denote $(\lambda_1,x_1)$ by $(\lambda,x)$. There are a number of
methods for computing $(\lambda,x)$, such as the inverse iteration
\cite{ParlettSEP}, the Rayleigh quotient iteration (RQI)
\cite{ParlettSEP}, the Lanczos method and its shift-invert variant
\cite{ParlettSEP}, the Davidson method and the Jacobi--Davidson
method \cite{stewart,vorst}. However, except the standard Lanczos
method, these methods and shift-invert Lanczos require the exact
solution of a possibly ill-conditioned linear system at each
iteration. This is generally very difficult and even impractical by
a direct solver since a factorization of a shifted $A$ may be too
expensive. So one generally resorts to iterative methods to solve
the linear systems involved, called inner iterations. We call
updates of approximate eigenpairs outer iterations. A combination of
inner and outer iterations yields an inner-outer iterative eigensolver,
also called an inexact eigensolver.

Among the inexact eigensolvers available, the inexact inverse iteration
and the inexact RQI are the simplest and most basic ones. They not
only have their own rights but also are key ingredients of other more
sophisticated and practical inexact solvers, such as inverse
subspace iteration \cite{robbe}
and the Jacobi--Davidson method. So one must first
analyze their convergence. This is generally the first step towards
better understanding and analyzing other more practical inexact
solvers.

For $A$ Hermitian or non-Hermitian, the inexact inverse iteration
and the inexact RQI have been considered, and numerous convergence
results have been established in many papers, e.g.,
\cite{mgs06,MullerVariableShift,freitagspence,
golubye,hochnotay,Lailin,NotayRQI,SimonciniRQI,Smit,EshofJD,xueelman}
and the references therein. For the Hermitian eigenproblem, general
theory for the inexact RQI can be found in Berns-M\"uller and
Spence~\cite{mgs06}, Smit and Paadekooper \cite{Smit} and van den
Eshof~\cite{EshofJD}. They prove that the inexact RQI achieves cubic
and quadratic convergence with decreasing inner tolerance
$\xi_k=O(\|r_k\|)$ and $\xi_k\leq\xi<1$ with a constant $\xi$ not
near one. Supposing that the shifted linear systems are solved by
the minimal residual method (MINRES) \cite{paige,saad},
mathematically equivalent to the conjugate residual method
\cite{saad}, Simoncini and Eld$\acute{e}$n~\cite{SimonciniRQI} prove
the cubic and quadratic convergence of the inexact RQI with MINRES
and present a number of important results under the same assumption
on $\xi_k$. Simoncini and Eld$\acute{e}$n first observed that the
convergence of the inexact RQI may allow $\xi_k$ to almost stagnate,
that is, i.e., $\xi_k$ near one. Xue and Elman~\cite{xueelman} have
refined and extended some results due to Simoncini and
Eld$\acute{e}$n. They have proved that MINRES typically exhibits a
very slow residual decreasing property (i.e., stagnation in their
terminology) during initial steps but the inexact RQI may still
converge; it is the smallest harmonic Ritz value that determines the
convergence of MINRES for the shifted linear systems.
Furthermore, although Xue and Elman's results have indicated
that very slow MINRES residual
decreasing {\em may not} prevent the convergence of inexact RQI, too
slow residual decreasing does matter and will make it fail to
converge. Besides, for the inexact RQI, to the
author's best knowledge, there is no result available on linear
convergence and its conditions.

In this paper we first study the convergence of the inexact RQI,
independent of iterative solvers for inner linear systems. We
present new general convergence results under a certain
uniform positiveness condition, which takes into account the residual
directions obtained by iterative solvers for the inner linear
systems. Unlike the common condition $\xi_k\leq\xi<1$ with $\xi$ a constant,
it appears that the uniform positiveness condition critically
depends on iterative solvers for inner iterations and may allow
$\xi_k\approx 1$ and even $\xi_k>1$, much weaker than the common condition
in existing literatures.

We then focus on the inexact RQI with the unpreconditioned MINRES
used for solving inner shifted linear systems. Our key observation
is that one usually treats the residuals obtained by MINRES as general ones,
simply takes their norms but ignores their directions. As will be clear
from our general convergence results, residual directions of the
linear systems play a crucial role in refining convergence analysis of
the inexact RQI with MINRES. We first establish a few attractive properties of
the residuals obtained by MINRES for the shifted linear systems. By
combining them with the new general convergence theory, we derive a
number of new insightful results that are not only stronger than but
also fundamentally different from the known ones in the literature.
We show how the inexact RQI with MINRES meets the uniform
positiveness condition and how it behaves if the condition fails to
hold.

As will be clear, we trivially have $\xi_k\leq 1$ for MINRES at any
inner iteration step. We prove that the inexact RQI with MINRES
generally converges cubically if the uniform positiveness condition
holds. This condition is shown to be equivalent to $\xi_k\leq \xi$
with a fixed $\xi$ not near one, but the inexact RQI with MINRES now
has cubic convergence other than the familiar quadratic convergence.
Cubic convergence does not require decreasing inner tolerance
$\xi_k=O(\|r_k\|)$ any more. We will see that $\xi=0.1,\ 0.5$ work
are enough, $\xi=0.8$ works well and a smaller $\xi$ is not
necessary. We prove that quadratic convergence only requires
$\xi_k=1-O(\|r_k\|)$, which tends to one as $\|r_k\|\rightarrow 0$
and is much weaker than the familiar condition $\xi_k\leq\xi$ with a
constant $\xi$ not near one. Besides, we show that a linear
convergence condition is $\xi_k= 1-O(\|r_k\|^2)$, closer to one than
$1-O(\|r_k\|)$. Therefore, if stagnation occurs during inner
iterations, the inexact RQI may converge quadratically or linearly;
if stagnation is too serious, that is, if $\xi_k$ is closer to one
than $1-O(\|r_k\|^2)$, the method may fail to converge. Note that,
for the inner linear systems, the smaller $\xi_k$ is, the more
costly it is to solve them using MINRES. As a result, in order to
achieve cubic and quadratic convergence, our new conditions are more
relaxed and easier to meet than the corresponding known ones.
Therefore, our results not only give new insights into the method
but also have impacts on its effective implementations. They allow
us to design practical criteria to best control inner tolerance to
achieve a desired convergence rate and to implement the method more
effectively than ever before. Numerical experiments demonstrate
that, in order to achieve cubic convergence, our new implementation
is about twice as fast as the original one with $\xi_k=O(\|r_k\|)$.

Besides, we establish a lower bound on the norms of approximate
solutions $w_{k+1}$'s of the linear systems obtained by MINRES. We
show that they are of $O(\frac{1}{\|r_k\|^2})$,
$O(\frac{1}{\|r_k\|})$ and $O(1)$ when the inexact MINRES converges
cubically, quadratically and linearly, respectively. So
$\|w_{k+1}\|$ can reflect how fast the inexact RQI converges and can
be used to control inner iteration, similar to those done in, e.g.,
\cite{SimonciniRQI,xueelman}. Making use of the bound, as a
by-product, we present a simpler but weaker convergence result on
the inexact RQI with MINRES. It and the bound for $\|w_{k+1}\|$ are
simpler and interpreted more clearly and easily than those obtained
by Simoncini and Eld$\acute{e}$n~\cite{SimonciniRQI}. However, we
will see that our by-product and their result are weaker than our
main results described above. An obvious drawback is that the cubic
convergence of the exact RQI and of the inexact RQI with MINRES
cannot be recovered when $\xi_k=0$ and $\xi_k=O(\|r_k\|)$,
respectively.

It appears \cite{freitagspence,freitag08b,xueelman} that it is often
beneficial to precondition each shifted inner linear system with a
tuned preconditioner, which can be much more effective than the corresponding
usual preconditioner. How to extend the main results on the inexact RQI with
the unpreconditioned MINRES case to the inexact RQI with
a tuned preconditioned MINRES turns out to be nontrivial. We will carry out
this task in the paper.

The paper is organized as follows. In Section \ref{SecIRQI}, we
review the inexact RQI and present new general convergence results,
independent of iterative solvers for the linear systems. In
Section~\ref{secminres}, we present cubic, quadratic and linear
convergence results on the inexact RQI with the unpreconditioned MINRES.
In Section~\ref{precondit}, we show that
the theory can be extended to the tuned preconditioned MINRES case.
We perform numerical experiments to confirm our results in
Section~\ref{testminres}. Finally, we end up with some concluding
remarks in Section~\ref{conc}.

Throughout the paper, denote by the superscript
* the conjugate transpose of a matrix or vector, by $I$ the
identity matrix of order $n$, by $\|\cdot\|$ the vector 2-norm and
the matrix spectral norm, and by $\lambda_{\min},\lambda_{\max}$
the smallest and largest eigenvalues of $A$, respectively.

\section{The inexact RQI and general convergence theory}\label{SecIRQI}

RQI is a famous iterative algorithm and its locally cubic
convergence for Hermitian problems is very attractive
\cite{ParlettSEP}. It plays a crucial role in some practical
effective algorithms, e.g., the QR algorithm,
\cite{GolubMC,ParlettSEP}. Assume that the unit length $u_k$ is already a
reasonably good approximation to $ x $. Then the Rayleigh quotient
$\theta_k = u^*_k A u_k$ is a good approximation to $\lambda$ too.
RQI computes a new approximation $u_{k+1}$ to $x$ by solving the
shifted inner linear system
\begin{equation} \label{EqERQILinearEquation1}
( A - \theta_k I ) w = u_k
\end{equation}
for $w_{k+1}$ and updating $u_{k+1}=w_{k+1}/\|w_{k+1}\|$ and
iterates until convergence. It is known
\cite{mgs06,NotayRQI,ParlettSEP} that if
$$
|\lambda-\theta_0|
<\frac{1}{2}\min_{j=2,3,\ldots,n}|\lambda-\lambda_j|
$$
then RQI (asymptotically) converges to $ \lambda $ and $x$
cubically. So we can assume that the eigenvalues of $A$ are ordered
as
\begin{equation}
|\lambda-\theta_k| <|\lambda_2-\theta_k| \le \cdots \le
|\lambda_n-\theta_k|. \label{order}
\end{equation}
With this ordering and $\lambda_{\min}\leq\theta_k\leq\lambda_{\max}$,
we have
\begin{equation}
|\lambda-\theta_k| <\frac{1}{2}|\lambda-\lambda_2|. \label{sep1}
\end{equation}

An obvious drawback of RQI is that at each iteration $k$ we need the
exact solution $w_{k+1}$ of $( A - \theta_k I ) w= u_k$. For a large
$A$, it is generally very expensive and even impractical to solve it
by a direct solver due to excessive memory and/or computational
cost. So we must resort to iterative solvers to get an approximate
solution of it. This leads to the inexact RQI.
\eqref{EqERQILinearEquation1} is solved by an iterative solver
and an approximate solution $w_{k+1}$ satisfies
\begin{equation} \label{EqIRQILinearEquation1}
( A - \theta_k I )w_{k+1} = u_k + \xi_k d_k, \quad u_{k+1} =w_{k+1}
/ \|w_{k+1} \|
\end{equation}
with $ 0 < \xi_k \le \xi$, where $\xi_kd_k $ with $\|d_k\|=1$ is the
residual of $(A-\theta_k I)w=u_k$, $d_k$ is the residual direction
vector and $\xi_k$ is the {\em relative} residual norm (inner
tolerance) as $\|u_k\|=1$ and may change at every outer iteration
$k$. This process is summarized as Algorithm 1. If $\xi_k=0$ for all
$k$, Algorithm 1 becomes the exact RQI.

\begin{algorithm}
  \caption{The inexact RQI} \label{AlgIRQI}
  \begin{algorithmic}[1]
  \STATE Choose a unit length $u_0$, an approximation to $x$.
  \FOR{$k$=0,1, \ldots}
  \STATE $\theta_k = u^*_k A u_k$.
  \STATE Solve $(A-\theta_k I)w=u_k$ for $w_{k+1}$ by an
  iterative solver with
  $$
  \|(A-\theta_k I)
  w_{k+1}-u_k\|=\xi_k\leq\xi.
  $$
  \STATE $u_{k+1} = w_{k+1}/\|w_{k+1}\|$.
  \STATE If convergence occurs, stop.
  \ENDFOR
  \end{algorithmic}
\end{algorithm}

It is always assumed in the literature that $\xi<1$ when making a
convergence analysis. This requirement seems very natural as
heuristically \eqref{EqERQILinearEquation1} should be solved with
some accuracy. Van den Eshof \cite{EshofJD} presents a quadratic
convergence bound that requires $\xi$ not near one, improving a
result of \cite{Smit} by a factor two. Similar quadratic convergence
results on the inexact RQI have also been proved in some other
papers, e.g., \cite{mgs06,SimonciniRQI,Smit}, under the same
condition on $\xi$. To see a fundamental difference between the
existing results and ours (cf.
Theorem~\ref{ThmIRQIQuadraticConvergence}), we take the result of
\cite{EshofJD} as an example and restate it. Before proceeding, we
define
$$
\beta=\frac{\lambda_{\max}-\lambda_{\min}}{|\lambda_2-\lambda|}
$$
throughout the paper. We comment that
$\lambda_{\max}-\lambda_{\min}$ is the spectrum spread of $A$ and
$|\lambda_2-\lambda|$ is the gap or separation of $\lambda$ and the
other eigenvalues of $A$.

\begin{theorem}{\rm \cite{EshofJD}} \label{eshof}
Define $\phi_k=\angle(u_k, x)$ to be the acute angle between $u_k$
and $x$, and assume that $w_{k+1}$ is such that
\begin{equation} \label{EqIRQIXi}
\| ( A - \theta_k I ) w_{k+1} - u_k \|=\xi_k \le \xi<1.
\end{equation}
Then letting $\phi_{k+1}=\angle(u_{k+1},x)$ be the acute between
$u_{k+1}$ and $x$, the inexact RQI converges quadratically:
\begin{equation} \label{EqIRQIVan den Eshof}
\tan \phi_{k+1} \le \frac{\beta\xi}{\sqrt{ 1 - \xi^2 }} \sin^2
\phi_k + O( \sin^3 \phi_k ).
\end{equation}
\end{theorem}

However, we will see soon that the condition $\xi<1$ not near one
can be stringent and unnecessary for quadratic convergence. To see
this, let us decompose $ u_k $ and $ d_k $ into the orthogonal
direct sums
\begin{eqnarray}
&u_k = x \, \cos \phi_k + e_k \, \sin \phi_k, \quad e_k \perp x,
\label{EqIRQIDecompositoinOfu_k} \\
&d_k = x \, \cos \psi_k + f_k \, \sin \psi_k, \quad f_k \perp x
\label{EqIRQIDecompositoinOfd_k}
\end{eqnarray}
with $\|e_k\|=\|f_k\|=1$ and $\psi_k=\angle(d_k,x)$. Then
\eqref{EqIRQILinearEquation1} can be written as
\begin{eqnarray} \label{EqIRQILinearEquation2}
( A - \theta_k I ) w_{k+1} = ( \cos \phi_k + \xi_k \, \cos \psi_k ) \, x
+ ( e_k \, \sin \phi_k + \xi_k \, f_k \, \sin \psi_k ).
\end{eqnarray}
Inverting $ A - \theta_k I$ gives
\begin{equation} \label{EqIRQIwk}
w_{k+1} = ( \lambda - \theta_k )^{-1} ( \cos \phi_k + \xi_k \,
\cos \psi_k ) \, x + ( A - \theta_k I )^{-1} ( e_k \, \sin \phi_k
+ \xi_k \, f_k \, \sin \psi_k ).
\end{equation}

We now revisit the convergence of the inexact RQI and prove that it
is the size of $|\cos\phi_k + \xi_k\cos \psi_k|$ other than
$\xi_k\leq\xi<1$ that is critical in affecting convergence.

\begin{theorem} \label{ThmIRQIQuadraticConvergence}
If the uniform positiveness condition
\begin{equation} \label{EqIRQIC}
|\cos \phi_k + \xi_k \cos \psi_k| \ge c
\end{equation}
is satisfied with a constant $c>0$ uniformly independent of $k$,
then
\begin{eqnarray}
\tan \phi_{k+1} &\le &2\beta\frac{\sin \phi_k + \xi_k \sin \psi_k}
{|\cos \phi_k + \xi_k \cos \psi_k|}\sin^2\phi_k   \label{bound1}\\
&\le &\frac{2\beta\xi_k}{c} \sin^2 \phi_k + O( \sin^3 \phi_k ),
\label{EqIRQIQuadraticConvergence}
\end{eqnarray}
that is, the inexact RQI converges quadratically at least for
uniformly bounded $\xi_k\leq\xi$ with $\xi$ some moderate constant.
\end{theorem}

\begin{proof}
Note that (\ref{EqIRQIwk}) is an orthogonal direct sum decomposition
of $w_{k+1}$ since for a Hermitian $A$ the second term is orthogonal
to $x$. We then have
$$
\tan \phi_{k+1} = | \lambda - \theta_k | \frac{\| ( A - \theta_k I
)^{-1} ( e_k \, \sin \phi_k + \xi_k f_k \, \sin \psi_k ) \|} {|\cos
\phi_k + \xi_k \, \cos \psi_k|}.
$$
As $A$ is Hermitian and $e_k\perp x$, it is easy to verify (cf.
\cite[p. 77]{ParlettSEP}) that
$$
\lambda -\theta_k=(\lambda - e^*_k A e_k)\sin^2\phi_k,
$$
$$
|\lambda_2 - \lambda| \le | \lambda - e^*_k A e_k | \le
\lambda_{\max}-\lambda_{\min},
$$
\begin{equation}
|\lambda_2-\lambda|\sin^2\phi_k\leq |\lambda - \theta_k| \leq
(\lambda_{\max}-\lambda_{\min})\sin^2\phi_k. \label{error1}
\end{equation}
Since {\small
\begin{eqnarray*}
\|( A -\theta_k I )^{-1} ( e_k \,
\sin \phi_k + \xi_k f_k \, \sin \psi_k )\|& \le& \| ( A - \theta_k
I )^{-1} e_k \| \sin \phi_k +
\xi_k \| ( A - \theta_k I )^{-1} f_k \| \sin \psi_k \\
& \le& |\lambda_2 -\theta_k|^{-1}(\sin \phi_k + \xi_k \sin\psi_k
)\\
&\leq&2 |\lambda_2-\lambda|^{-1}(\sin \phi_k + \xi_k \sin\psi_k)
\end{eqnarray*}}
with the last inequality holding because of (\ref{order}), we get
\begin{eqnarray*}
\tan \phi_{k+1}& \le& | \lambda - e^*_k A e_k | \sin^2 \phi_k \frac{
2(\sin \phi_k + \xi_k \sin \psi_k)} { |\lambda_2 - \lambda||\cos
\phi_k + \xi_k\cos
\psi_k|}\\
&\leq&\frac{2|\lambda_n -\lambda|}{|\lambda_2 -\lambda|}\frac{\sin
\phi_k + \xi_k \sin \psi_k} {|\cos \phi_k + \xi_k \cos
\psi_k|}\sin^2\phi_k
\\
&\leq&\xi_k \frac{2(\lambda_{\max}-\lambda_{\min})}{c|\lambda_2 -
\lambda|} \sin^2 \phi_k + O( \sin^3 \phi_k ).
\end{eqnarray*}
\end{proof}

Define $\|r_k\|=\|(A-\theta_k I)u_k\|$. Then by (\ref{sep1}) we get
$|\lambda_2-\theta_k|>\frac{|\lambda_2-\lambda|}{2}$. So it is known
from \cite[Theorem 11.7.1]{ParlettSEP} that
\begin{equation}
\frac{\|r_k\|}{\lambda_{\max}-\lambda_{\min}}\leq\sin\phi_k\leq\frac{2\|r_k\|}
{|\lambda_2-\lambda|}.\label{parlett}
\end{equation}

We can present an alternative of (\ref{EqIRQIQuadraticConvergence})
when $\|r_k\|$ is concerned.

\begin{theorem}\label{resbound}
If the uniform positiveness condition {\rm (\ref{EqIRQIC})} holds,
then
\begin{equation}
\|r_{k+1}\|\leq\frac{8\beta^2\xi_k} {c|\lambda_2-\lambda|}\|r_k\|^2
+O(\|r_k\|^3).\label{resgeneral}
\end{equation}
\end{theorem}

\begin{proof}
Note from (\ref{parlett}) that
$$
\frac{\|r_{k+1}\|}{\lambda_{\max}-\lambda_{\min}}\leq\sin\phi_{k+1}\leq\tan\phi_{k+1}.
$$
Substituting it and the upper bound of (\ref{parlett}) into
(\ref{EqIRQIQuadraticConvergence}) establishes (\ref{resgeneral}).
\end{proof}

For the special case that the algebraically smallest eigenvalue is
of interest, Jia and Wang~\cite{jiawang} proved a slightly different
result from Theorem~\ref{ThmIRQIQuadraticConvergence}. Similar to
the literature, however, they still assumed that $\xi_k\leq\xi<1$ and did not
analyze the theorem further, though their proof did not use this
assumption. A striking insight from the theorem is that the
condition $\xi_k\leq\xi<1$ may be considerably relaxed. If
$\cos\psi_k$ is positive, the uniform positiveness condition holds
for any uniformly bounded $\xi_k\leq\xi$. So we can have $\xi\geq 1$
considerably. If $\cos\psi_k$ is negative, then
$|\cos\phi_k+\xi_k\cos\psi_k|\geq c$ means that
$$
\xi_k\leq\frac{c-\cos\phi_k}{\cos\psi_k}
$$
if $\cos\phi_k+\xi_k\cos\psi_k\geq c$ with $c<1$ is required or
$$
\xi_k\geq\frac{c+\cos\phi_k}{-\cos\psi_k}
$$
if $-\cos\phi_k-\xi_k\cos\psi_k\geq c$ is required. Keep in mind
that $\cos\phi_k\approx 1$. So the size of $\xi_k$ critically depends on
that of $\cos\psi_k$, and for a given $c$ we may have $\xi_k\approx 1$ and even
$\xi_k>1$. Obviously, without the information on $\cos\psi_k$, it
would be impossible to access or estimate $\xi_k$. As a general
convergence result, however, its significance and importance consist
in that it reveals a new remarkable fact: It appears first time that (13)
may allow $\xi_k$ to be relaxed (much) more than that used in all
known literatures and meanwhile preserves the same convergence rate
of outer iteration. As a result, the condition $\xi_k\leq\xi<1$ with
constant $\xi$ not near one may be stringent and unnecessary for the
quadratic convergence of the inexact RQI, independent of iterative
solvers for the linear systems. The new condition has a strong impact on
practical implementations as we must use a
certain iterative solver, e.g., the very popular MINRES method and
the Lanczos method (SYMMLQ) for solving $(A-\theta_k I)w=u_k$. We will
see that $\cos\psi_k$ is critically iterative solver dependent. For
MINRES, $\cos\psi_k$ has some very attractive properties,
by which we can precisely determine bounds for $\xi_k$ in
Section~\ref{secminres},
which are much more relaxed than those in the literature. For the Lanczos
method, we refer to \cite{jia09} for $\cos\psi_k$ and $\xi_k$, where
$\cos\psi_k$ and $\xi_k$ are fundamentally different from those
obtained by MINRES and $\xi_k\geq 1$ considerably is allowed.

{\bf Remark 1}. If $ \xi_k = 0 $ for all $ k $, the inexact RQI
reduces to the exact RQI and
Theorems~\ref{ThmIRQIQuadraticConvergence}--\ref{resbound} show
(asymptotically) cubic convergence: $\tan \phi_{k+1}=O(\sin^3\phi_k)$ and
$\|r_{k+1}\|=O(\|r_k\|^3)$.

{\bf Remark 2}.  If the linear systems are solved with decreasing
tolerance $\xi_k=O(\sin\phi_k)=O(\|r_k\|)$, then  $\tan
\phi_{k+1}=O(\sin^3\phi_k)$ and $\|r_{k+1}\|=O(\|r_k\|^3)$. Such
(asymptotically) cubic convergence also appears in several papers, e.g.,
\cite{mgs06,Smit,EshofJD}, either explicitly or implicitly.

\section{Convergence of the inexact RQI with MINRES}\label{secminres}

The previous results and discussions are for general purpose,
independent of iterative solvers for $(A-\theta_k I)w=u_k$. Since
we have $\lambda_{\rm min}\leq\theta_k\leq\lambda_{\rm max}$,
the matrix $A-\theta_k I$ is Hermitian indefinite. One of the most
popular iterative solvers for $(A-\theta_k I)w=u_k$ is the MINRES
method as it has a very attractive residual monotonic decreasing
property \cite{paige,saad}. This leads to the inexact RQI with MINRES.

We briefly review MINRES for solving (\ref{EqERQILinearEquation1}).
At outer iteration $k$, taking the starting vector $v_1$ to be
$u_k$, the $m$-step Lanczos process on $A-\theta_k I$ can be written
as
\begin{equation}
( A - \theta_k I ) V_m =
V_mT_m+t_{m+1m}v_{m+1}e_m^*=V_{m+1}\hat{T}_m, \label{lanczosp}
\end{equation}
where the columns of $ V_m =(v_1,\ldots,v_m)$ form an orthonormal
basis of the Krylov subspace $\mathcal{K}_m(A - \theta_k I, u_k )
= \mathcal{K}_m ( A, u_k )$, $V_{m+1}=(V_m, v_{m+1})$,
$T_m=(t_{ij})=V_m^*(A-\theta_kI)V_m$ and
$\hat{T}_m=V_{m+1}^*(A-\theta_kI)V_m$ \cite{ParlettSEP,saad}.
Taking the zero vector as an initial guess to the solution of
$(A-\theta_k I)w=u_k$, MINRES \cite{GolubMC,paige,saad} extracts the
approximate solution $w_{k+1}=V_m\hat y$ to $(A-\theta_k I)w=u_k$
from $\mathcal{K}_m (A, u_k)$, where $\hat y$ is the solution of the
least squares problem $\min\|e_1-\hat{T}_my\|$ with $e_1$ being the
first coordinate vector of dimension $m+1$.

By the residual monotonic decreasing property of MINRES, we
trivially have $\xi_k \le\|u_k \| = 1 $ for all $k$ and any inner
iteration steps $m$. Here we must take $m>1$; for $m=1$,
it is easily verified that $\hat y=0$ and thus $w_{k+1}=0$ and
$\xi_k=1$ by noting that $t_{11}=u_k^*(A-\theta_k I)u_k=0$. So
$u_{k+1}$ is undefined and the inexact RQI with MINRES breaks down
if $m=1$.

The residual direction vectors $d_k$ obtained by MINRES
have some attractive features and we can precisely get subtle bounds
for $\sin\psi_k$ and $\cos\psi_k$, as the following results show.

\begin{theorem}\label{minres}
For MINRES, let the unit length vectors $e_k$ and $f_k$ be as in
{\rm (\ref{EqIRQIDecompositoinOfu_k})} and {\rm
(\ref{EqIRQIDecompositoinOfd_k})}, define the angle
$\varphi_k=\angle(f_k,(A-\theta_k I)e_k)$ and assume that
\begin{equation}
|\cos\varphi_k|\geq|\cos\varphi|>0 \label{assump1}
\end{equation}
holds uniformly for an angle $\varphi$ away from $\frac{\pi}{2}$
independent of $k$, i.e.,
\begin{equation}
|f_k^*(A-\theta I)e_k|=\|(A-\theta_k I)e_k\||\cos\varphi_k|\geq
\|(A-\theta_k I)e_k\||\cos\varphi|.\label{assump2}
\end{equation}
Then we have
\begin{eqnarray}
 \sin\psi_k&\leq&\frac{2\beta}{|\cos\varphi|}
 \sin\phi_k,\label{sinpsi}\\
 \cos\psi_k&=&\pm (1-O(\sin^2\phi_k)). \label{cospsi}
\end{eqnarray}
Furthermore, provided $\xi_k>\sin\phi_k$, then
\begin{eqnarray}
\cos\psi_k&=&-1+O(\sin^2\phi_k),\label{cospsik}\\
d_k&=&-x+O(\sin\phi_k).\label{dkerror}
\end{eqnarray}
\end{theorem}

\begin{proof}
Note that for MINRES its residual $\xi_kd_k$ satisfies
$\xi_kd_k\perp (A-\theta_k I)\mathcal{K}_m(A,u_k)$. Therefore, we
specially have $\xi_k d_k \perp (A-\theta_k I) u_k$, i.e., $d_k
\perp (A-\theta_k I) u_k$. Then from
(\ref{EqIRQIDecompositoinOfu_k}) and
(\ref{EqIRQIDecompositoinOfd_k}) we obtain
$$
(\lambda-\theta_k)\cos\phi_k\cos\psi_k+f_k^*(A-\theta_kI)e_k\sin\phi_k\sin\psi_k=0.
$$
So
\begin{equation}
\tan\psi_k=\frac{(\theta_k-\lambda)\cos\phi_k}{f_k^*(A-\theta_kI)e_k\sin\phi_k}.
\label{tanpsi}
\end{equation}
By $|\lambda-\theta_k|< \frac{|\lambda-\lambda_2|}{2}$, we get
\begin{eqnarray*}
|f_k^*(A-\theta_k I)e_k|&=&\|(A-\theta _k I)e_k\||\cos\varphi_k|\\
&\geq& \|(A-\theta _k I)e_k\||\cos\varphi|\\
&\geq&|\lambda_2-\theta_k||\cos\varphi|\\
&>&\frac{|\lambda_2-\lambda|}{2}|\cos\varphi|.
\end{eqnarray*}
Using (\ref{error1}), we obtain from (\ref{tanpsi})
\begin{eqnarray*}
|\tan\psi_k|&\leq&\frac{(\lambda_{\max}-\lambda_{\min})\sin\phi_k\cos\phi_k}
{|f_k^*(A-\theta_k I)e_k|}\\
&\leq&\frac{2(\lambda_{\max}-\lambda_{\min})}{|\lambda_2-\lambda||\cos\varphi|}
\sin\phi_k\cos\phi_k\\
&\leq&\frac{2\beta}{|\cos\varphi|}\sin\phi_k.
\end{eqnarray*}
Therefore, (\ref{sinpsi}) holds. Note that (\ref{sinpsi}) means
$\sin\psi_k=O(\sin\phi_k)$. So we get
$$
\cos\psi_k=\pm\sqrt{1-\sin^2\psi_k}=\pm
(1-\frac{1}{2}\sin^2\psi_k)+O(\sin^4\psi_k)=\pm (1-O(\sin^2\phi_k))
$$
by dropping the higher order term $O(\sin^4\phi_k)$.

Now we prove that  $\cos\psi_k$ and $\cos\phi_k$ must have opposite
signs if $\xi_k>\sin\phi_k$. Since the MINRES residual
$$
\xi_kd_k=(A-\theta_k I)w_{k+1}-u_k,
$$
by its residual minimization property we know that $(A-\theta_k
I)w_{k+1}$ is just the orthogonal projection of $u_k$ onto
$(A-\theta_k I){\cal K}_m(A,u_k)$ and $\xi_kd_k$ is orthogonal to
$(A-\theta_k I)w_{k+1}$. Therefore, we get
\begin{equation}
\xi_k^2+\|(A-\theta_k I)w_{k+1}\|^2=\|u_k\|^2=1.\label{minpro}
\end{equation}

Note that
$$
(A-\theta_k I)w_{k+1}=u_k+\xi_k d_k=(\cos\phi_k+\xi_k\cos\psi_k)x
+(e_k\sin\phi_k+\xi_kf_k\sin\psi_k)
$$
is an orthogonal direct sum decomposition of $(A-\theta_k
I)w_{k+1}$. Therefore, we have
$$
\|(A-\theta_k I)w_{k+1}\|^2=(\cos\phi_k+\xi_k\cos\psi_k)^2
+\|e_k\sin\phi_k+\xi_kf_k\sin\psi_k\|^2,
$$
which, together with (\ref{minpro}), gives
$$
(\cos\phi_k+\xi_k\cos\psi_k)^2+\xi_k^2\leq 1.
$$
Solving it for $\xi_k\cos\psi_k$, we have
$$
-\cos\phi_k-\sqrt{1-\xi_k^2}\leq \xi_k\cos\psi_k\leq
-\cos\phi_k+\sqrt{1-\xi_k^2},
$$
in which the upper bound is negative provided that
$-\cos\phi_k+\sqrt{1-\xi_k^2}<0$, which means $\xi_k>\sin\phi_k$. So
$\cos\psi_k$ and $\cos\phi_k$ must have opposite signs if
$\xi_k>\sin\phi_k$. Hence, it follows from \eqref{cospsi} that
\eqref{cospsik} holds if $\xi_k>\sin\phi_k$. Combining
(\ref{cospsi}) with \eqref{EqIRQIDecompositoinOfd_k} and
\eqref{sinpsi} gives \eqref{dkerror}.
\end{proof}

Since $u_k$ is assumed to be a reasonably good approximation to $x$,
$\sin\phi_k$ is small. As a result, the condition $\xi_k>\sin\phi_k$
is easily satisfied unless the linear system is solved with very
high accuracy.

Clearly, how general this theorem depends on how general assumption
(\ref{assump1}) is. Next we give a qualitative analysis to show
that this assumption is very reasonable and holds generally, indicating that
the theorem is of rationale and generality.

Note that by definition we have
$$
e_k \mbox{ and }f_k\in {\rm span}\{x_2,x_3,\ldots,x_n\},
$$
so does
$$
(A-\theta_kI)e_k\in {\rm span}\{x_2,x_3,\ldots,x_n\}.
$$
We now justify the generality of $e_k$ and $f_k$. In the proof of
cubic convergence of RQI, which is the inexact RQI with $\xi_k=0$,
Parlett \cite[p. 78-79]{ParlettSEP} proves that $e_k$ will start to
converge to $x_2$ only after $u_k$ has converged to $x$ and
$e_k\rightarrow x_2$ holds for large enough $k$. In other words,
$e_k$ is a general combination of $x_2,x_3,\ldots,x_n$ and does not
start to converge before $u_k$ has converged. Following his proof
path, we have only two possibilities on $e_k$ in the inexact RQI
with MINRES: One is that $e_k$, at best, can possibly start to approach $x_2$
only if $u_k$ has converged to $x$; the other is that $e_k$ is nothing
but just still a general linear combination of $x_2,x_3,\ldots,x_n$ and
does not converge to any specific vector for any $\xi_k$. In either
case, $e_k$ is indeed a general linear combination of $x_2,x_3,\ldots,x_n$
before $u_k$ has converged.

Expand the unit length $e_k$ as
$$
e_k=\sum_{j=2}^n\alpha_jx_j
$$
with $\sum_{j=2}^n\alpha_j^2=1$. Then, based on the above arguments,
no $\alpha_j$ is small before $u_k$ has converged. Note that
$$
(A-\theta_k I)e_k=\sum_{j=2}^n\alpha_j (\lambda_j-\theta_k)x_j.
$$
Since $\theta_k$ supposed to be a reasonably good approximation to
$\lambda$, for $j=2,3,\ldots,n$, $\lambda_j-\theta_k$ are not small,
so that $(A-\theta_k I)e_k$ is a general linear combination of
$x_2,x_3,\ldots,x_n$.

Let $p_m(z)$ be the residual polynomial of MINRES applied to
(\ref{EqERQILinearEquation1}). Then it is known \cite{paige} that
$p_m(0)=1$, its $m$ roots are the harmonic values of $A-\theta_k I$
with respect to ${\cal K}_m(A,u_k)$. From
(\ref{EqIRQIDecompositoinOfu_k}) and
(\ref{EqIRQIDecompositoinOfd_k}) we can write the residual
$\xi_kd_k$ as
\begin{eqnarray*}
\xi_k d_k&=&p_m(A-\theta_k I)u_k=p_m(A-\theta_k I)(x\cos\phi_k+
e_k\sin\phi_k)\\
&=& \cos\phi_kp_m(\lambda-\theta_k)x+\sin\phi_k
\sum_{j=2}^n\alpha_jp_m(\lambda_j-\theta_k)x_j\\
&=&\xi_k(x\cos\psi_k+f_k\sin\psi_k).
\end{eqnarray*}
Noting that $\|f_k\|=1$, we get
$$
f_k=\frac{p_m(A-\theta_k
I)e_k}{\|p_m(A-\theta_kI)e_k\|}=\frac{\sum_{j=2}^n\alpha_j
p_m(\lambda_j-\theta_k)x_j}
{(\sum_{j=2}^n\alpha_j^2p_m^2 (\lambda_j-\theta_k))^{1/2}}.
$$
Since $u_k$ is rich in $x$ and has small components in
$x_2,x_3,\ldots,x_n$, ${\cal K}_m(A,u_k)$ contains not much
information on $x_2,x_3,\ldots,x_n$ unless $m$ is large enough.
Therefore, as approximations to the eigenvalues
$\lambda_2-\theta_k,\lambda_3-\theta_k,\ldots,\lambda_n-\theta_k$ of
the matrix $A-\theta_kI$, the harmonic Ritz values are generally of
poor quality unless $m$ is large enough. By continuity,
$p_m(\lambda_j-\theta_k),\,j=2,3,\ldots,n$ are generally not near
zero. This means that usually $f_k$ is a general linear combination of
$x_2,x_3,\ldots,x_n$, which is the case for $\xi_k$ not very small.
We point out that $p_m(\lambda_j - \theta_k),\ j = 2,3,...,n$ can be
possibly near zero only if $\xi_k$ is small enough, but in this case
the cubic convergence of the outer iteration can be established
trivially, according to Theorem~\ref{ThmIRQIQuadraticConvergence}.
For other discussions on $\xi_k$ and harmonic Ritz values, we refer
to \cite{xueelman}.

In view of the above, it is very unlikely for $f_k$ and $(A-\theta_k
I)e_k$ to be nearly orthogonal, that is, $\varphi_k$ is rarely near
$\frac{\pi}{2}$. So, $|\cos\varphi_k|$ should be uniformly away from
zero in general, and assumption (\ref{assump1}) is very general and
reasonable.

Obviously, $|\cos\varphi_k|$ is a-priori and cannot be computed in
practice. For test purposes, for each matrix in
Section~\ref{testminres} and some others, supposing the $x$'s are
known, we have computed $|\cos\varphi_k|$ for $\xi_k=O(\|r_k\|)$,
$\xi_k\leq \xi<1$ with $\xi$ a fixed constant, $\xi_k=1-(\|r_k\|)$
and $\xi_k=1-O(\|r_k\|^2)$, respectively. As will be seen, the
latter three requirements on $\xi_k$ are our new cubic, quadratic
and linear convergence conditions for the inexact RQI with MINRES
that will be derived by combining this theorem with
Theorems~\ref{ThmIRQIQuadraticConvergence}--\ref{resbound} and are
critically need the assumption that $|\cos\varphi_k|$ is uniformly
away from zero, independent of $k$. Among thousands
$|\cos\varphi_k|$'s, we have found that most of the
$|\cos\varphi_k|$'s are considerably away from zero and very few
smallest ones are around $10^{-4}$. Furthermore, we have found
that their arithmetic mean is basically $0.016\sim 0.020$ for each
matrix and a given choice of $\xi_k$ above. To highlight these
results and to be more illustrative, we have extensively computed
$|\cos\angle(w,v)|=\frac{|w^*v|}{\|w\|\|v\|}$ with many vector pairs
$(w,v)$ of numerous dimensions no less than 1000 that were generated
randomly in a normal distribution. We have observed that the values
of our $|\cos\varphi_k|$'s and their arithmetic mean for each matrix
and a given choice of $\xi_k$ above have very similar behavior to
those of thousands $|\cos\angle(w,v)|$'s and the arithmetic mean
for a given dimension. As a
consequence, this demonstrates that assumption~(\ref{assump1}) are
very like requiring the uniform non-orthogonality of two normal
random vectors and thus should hold generally. In other words,
$f_k,e_k$ and $(A-\theta_k I)e_k$ are usually indeed general
linear combinations of $x_2,\ldots,x_n$ and have the nature as vectors
generated in a normal distribution, so assumption~(\ref{assump1})
should hold generally. On the other hand, our later
numerical experiments also confirm the cubic, quadratic and linear
convergence under our corresponding new conditions  This means that
the numerical experiments also justify the rationale and generality
of the assumption a-posteriori.

%In finite precision arithmetic, since $u_k\rightarrow x$,
%$\cos\phi_k\rightarrow 1$ and $\sin\phi_k\rightarrow 0$, it is seen
%from (\ref{EqIRQIDecompositoinOfu_k}) and
%(\ref{EqIRQIDecompositoinOfd_k}) that there is loss of accuracy in
%the computation of $e_k$ due to the cancelation and division by a
%small number in $e_k=(u_k-x\cos\phi_k)/\sin\phi_k$. Such loss of
%accuracy may also occur to the computation of
%$f_k=(u_k-x\cos\psi_k)/\sin\psi_k$. Therefore, due to round-off
%errors, one might think that computing
%$|\cos\varphi_k|=\frac{|f_k^*(A-\theta_k
%I)e_k|}{\|(A-\theta_kI)e_k\|}$ directly may not be a very reliable
%way to justify assumption (\ref{assump1}) numerically. However,
%provided that $\sin\phi_k,\sin\psi_k\gg\epsilon_{\rm mach}$, the machine
%precision, a simple round error analysis shows
%that such loss of accuracy has no essential effects on the value of
%$|\cos\varphi_k|$ once $\sin\phi_k,\sin\psi_k\gg \epsilon_{\rm
%mach}$. In fact, it can be shown that the difference between the
%computed $|\cos\varphi_k|$ and the true $|\cos\varphi_k|$ is of
%$\max\{O(\epsilon_{\rm mach}/\sin\phi_k),O(\epsilon_{\rm mach}/\sin\psi_k)\}$.

Based on Theorem~\ref{minres} and
Theorems~\ref{ThmIRQIQuadraticConvergence}--\ref{resbound},
we now present new convergence results on the inexact RQI with MINRES.

\begin{theorem}\label{minrescubic}
With $\cos\varphi$ defined as in Theorem~{\rm \ref{minres}},
assuming that $|\cos\varphi|$ is uniformly away from zero, the
inexact RQI with MINRES asymptotically converges cubically:
\begin{equation}
\tan\phi_{k+1}\leq\frac{2\beta
(|\cos\varphi|+2\xi_k\beta)}{(1-\xi_k)|\cos\varphi|}\sin^3\phi_k
\label{cubic}
\end{equation}
if $\xi_k\leq\xi$ with a fixed $\xi$ not near one; it converges
quadratically:
\begin{equation}
\tan\phi_{k+1}\leq\eta\sin^2\phi_k \label{quadminres}
\end{equation}
if $\xi_k$ is near one and bounded by
\begin{equation}
\xi_k\leq 1-\frac{6\beta^2\sin\phi_k}{\eta|\cos\varphi|}
\label{minresquad}
\end{equation}
with $\eta$ a moderate constant; it converges linearly at least:
\begin{equation}
\tan\phi_{k+1}\leq\zeta\sin\phi_k \label{linminres}
\end{equation}
with a constant $\zeta<1$ independent of $k$ if $\xi_k$ is bounded
by
\begin{equation}
1-\frac{6\beta^2\sin\phi_k}{\eta|\cos\varphi|}<\xi_k\leq
1-\frac{6\beta^2\sin^2\phi_k}{\zeta|\cos\varphi|}. \label{minreslin}
\end{equation}
\end{theorem}

\begin{proof}
Based on Theorem~\ref{minres}, we have
\begin{eqnarray}
\cos\phi_k+\xi_k\cos\psi_k&=&1-\frac{1}{2}\sin^2\phi_k\pm\xi_k
(1-O(\sin^2\phi_k))\nonumber\\
&=&1\pm\xi_k+O(\sin^2\phi_k)\nonumber\\
&=&1\pm\xi_k\geq 1-\xi_k \label{denominres}
\end{eqnarray}
by dropping the higher order term $O(\sin^2\phi_k)$. Therefore,
the uniform positiveness condition holds provided that
$\xi_k\leq\xi$ with a fixed $\xi$ not near one. Combining
(\ref{bound1}) with \eqref{sinpsi} and (\ref{cospsi}), we get
$$
\tan\phi_{k+1}\leq 2\beta\frac{1+\xi_k \frac{2\beta}{|\cos\varphi|}}
{1-\xi_k}\sin^3\phi_k,
$$
which is just (\ref{cubic}) and shows the cubic convergence of the
inexact RQI with MINRES if $\xi_k\leq \xi$ with a fixed $\xi$ not
near one.

Next we prove the quadratic convergence result. Since $\xi_k<1$ and
$\beta\geq 1$, it follows from (\ref{sinpsi}) and (\ref{denominres})
that asymptotically
\begin{eqnarray*}
2\beta\frac{\sin \phi_k + \xi_k \sin \psi_k} {|\cos \phi_k + \xi_k
\cos \psi_k|}&<&2\beta\frac{(1+2\beta/|\cos\varphi|)\sin\phi_k}
{\cos\phi_k+\xi_k\cos\psi_k}\\
&\leq&\frac{6\beta^2\sin\phi_k}
{(\cos\phi_k+\xi_k\cos\psi_k)|\cos\varphi|}\\
&\leq&\frac{6\beta^2\sin\phi_k}{(1-\xi_k)|\cos\varphi|}.
\end{eqnarray*}
So from (\ref{bound1}) the inexact RQI with MINRES converges
quadratically and (\ref{quadminres}) holds if
$$
\frac{6\beta^2\sin\phi_k} {(1-\xi_k)|\cos\varphi|}\leq\eta
$$
for a moderate constant $\eta$ independent of $k$. Solving this
inequality for $\xi_k$ gives (\ref{minresquad}).

Finally, we prove the linear convergence result. Analogously, we
have
$$
2\beta\frac{\sin\phi_k + \xi_k \sin\psi_k}{|\cos \phi_k + \xi_k \cos
\psi_k|}\sin\phi_k
<\frac{6\beta^2\sin^2\phi_k}{(1-\xi_k)|\cos\varphi|}.
$$
So it follows from (\ref{bound1}) that the inexact RQI with MINRES
converges linearly at least and (\ref{linminres}) holds when
$$
\frac{6\beta^2\sin^2\phi_k}{(1-\xi_k)|\cos\varphi|}\leq \zeta<1
$$
with a constant $\zeta$ independent of $k$. Solving the above
inequality for $\xi_k$ gives (\ref{minreslin}).
\end{proof}

Theorem~\ref{minrescubic} presents the convergence results in terms
of the a priori uncomputable $\sin\phi_k$. We next derive their
counterparts in terms of the a posteriori computable $\|r_k\|$, so
that they are of practical value as much as possible and can be used
to best control inner-outer accuracy to achieve a desired
convergence rate.

\begin{theorem}\label{minres2}
With $\cos\varphi$ defined as in Theorem~{\rm \ref{minres}},
assuming that $|\cos\varphi|$ is uniformly away from zero, the
inexact RQI with MINRES asymptotically converges cubically:
\begin{equation}
\|r_{k+1}\|\leq\frac{16\beta^2
(|\cos\varphi|+2\xi_k\beta)}{(1-\xi_k)(\lambda_2-\lambda)^2|\cos\varphi|}
\|r_k\|^3 \label{cubicres}
\end{equation}
if $\xi_k\leq\xi$ with a fixed $\xi$ not near one; it converges
quadratically:
\begin{equation}
\|r_{k+1}\|\leq\frac{4\beta\eta}{|\lambda_2-\lambda|}\|r_k\|^2
\label{quadres}
\end{equation}
if $\xi_k$ is near one and bounded by
\begin{equation}
\xi_k\leq 1-\frac{6\beta\|r_k\|}{\eta
|\lambda_2-\lambda||\cos\varphi|} \label{quadracond}
\end{equation}
with $\eta$ a moderate constant; it converges at linear factor
$\zeta$ at least:
\begin{equation}
\|r_{k+1}\|\leq\zeta\|r_k\| \label{linresmono}
\end{equation}
if $\xi_k$ is bounded by
\begin{equation}
1-\frac{6\beta\|r_k\|}{\eta
|\lambda_2-\lambda||\cos\varphi|}<\xi_k\leq
1-\frac{48\beta^3\|r_k\|^2}{\zeta(\lambda_2-\lambda)^2|\cos\varphi|}
\label{linrescond}
\end{equation}
with a constant $\zeta<1$ independent of $k$.
\end{theorem}

\begin{proof}
From (\ref{parlett}) we have
$$
\frac{\|r_{k+1}\|}{\lambda_{\max}-\lambda_{\min}}\leq\sin\phi_{k+1}\leq\tan\phi_{k+1},\
\sin\phi_k\leq\frac{2\|r_k\|}{|\lambda_2-\lambda|}.
$$
So (\ref{cubicres}) is direct from (\ref{cubic}) by a simple
manipulation. Next we use
(\ref{parlett}) to denote $\sin\phi_k=\frac{\|r_k\|}{C}$ with
$\frac{|\lambda_2-\lambda|}{2}\leq
C\leq\lambda_{\max}-\lambda_{\min}$. Note that
$$
2\beta\frac{\sin \phi_k + \xi_k \sin \psi_k} {|\cos \phi_k + \xi_k
\cos \psi_k|}<\frac{6\beta^2\sin\phi_k} {(1-\xi_k)|\cos\varphi|}=
\frac{6\beta^2\|r_k\|}{C(1-\xi_k)|\cos\varphi|}
$$
Therefore, similarly to the proof of Theorem~\ref{minrescubic},
if
$$
\frac{6\beta^2\|r_k\|}{C(1-\xi_k)|\cos\varphi|}\leq \eta,
$$
the inexact RQI with MINRES converges quadratically.
Solving this inequality for $\xi_k$ gives
\begin{equation}
\xi_k\leq
1-\frac{6\beta^2\|r_k\|}{C\eta|\cos\varphi|}.\label{boundxi}
\end{equation}
Note that
\begin{eqnarray*}
1-\frac{6\beta^2\|r_k\|}{C\eta|\cos\varphi|} &\geq&
1-\frac{6\beta^2\|r_k\|}{\eta(\lambda_{\max}-\lambda_{\min})|\cos\varphi|}\\
&=&1-\frac{6\beta\|r_k\|} {\eta|\lambda_2-\lambda||\cos\varphi|}.
\end{eqnarray*}
So, if $\xi_k$ satisfies (\ref{quadracond}), then it satisfies
(\ref{boundxi}) too. Therefore, the inexact RQI with MINRES
converges quadratically if (\ref{quadracond}) holds. Furthermore,
from (\ref{parlett}) we have
$\frac{\|r_{k+1}\|}{\lambda_{\max}-\lambda_{\min}}\leq\sin\phi_{k+1}\leq\tan\phi_{k+1}$.
As a result, from (\ref{quadminres}) we obtain
\begin{eqnarray*}
\|r_{k+1}\|&\leq&(\lambda_{\max}-\lambda_{\min})\eta\frac{\|r_k\|^2}{C^2}\\
&\leq&\frac{4\beta\eta}{|\lambda_2-\lambda|}\|r_k\|^2,
\end{eqnarray*}
proving (\ref{quadres}).

In order to make $\|r_k\|$ monotonically converge to zero linearly,
by (\ref{cubicres}) we simply set
$$
\frac{16\beta^2
(|\cos\varphi|+2\xi_k\beta)\|r_k\|^2}{(1-\xi_k)(\lambda_2-\lambda)^2|\cos\varphi|}
\leq\frac{48\beta^3\|r_k\|^2}{(1-\xi_k)
(\lambda_2-\lambda)^2|\cos\varphi|}\leq\zeta<1
$$
with $\zeta$ independent of $k$. Solving it for $\xi_k$ gives
$$
\xi_k\leq
1-\frac{48\beta^3\|r_k\|^2}{\zeta(\lambda_2-\lambda)^2|\cos\varphi|}.
$$
Combining it with (\ref{quadracond}) proves (\ref{linresmono}) and
(\ref{linrescond}).
\end{proof}

First of all, we make a comment on Theorems~\ref{minrescubic}--\ref{minres2}.
As justified previously, assumption (\ref{assump1}) is of wide
generality. Note that
Theorems~\ref{ThmIRQIQuadraticConvergence}--\ref{resbound} hold as
we always have $\cos\phi_k+\xi_k\cos\psi_k\geq 1-\xi_k$
asymptotically, independent of $\varphi_k$. Therefore, in case
$|\cos\varphi|$ is occasionally near and even zero, the inexact RQI
with MINRES converges quadratically at least provided that
$\xi_k\leq \xi$ with a fixed $\xi$ not near one.

In order to judge cubic convergence quantitatively, we should rely on
Theorem~\ref{ThmIRQIQuadraticConvergence} and
Theorem~\ref{minrescubic} (equivalently, Theorem~\ref{resbound} and
Theorem~\ref{minres2}), in which cubic convergence precisely means
that the asymptotic convergence factor
\begin{equation}
\frac{\sin\phi_{k+1}}{\sin^3\phi_k}\leq2\beta \label{rate1}
\end{equation}
for RQI and the asymptotic convergence factor
\begin{equation}
\frac{\sin\phi_{k+1}}{\sin^3\phi_k}\leq\frac{2\beta
(|\cos\phi|+2\xi_k\beta)}{(1-\xi_k)|\cos\varphi|}=
\frac{2\beta}{1-\xi_k}+\frac{4\xi_k\beta^2} {(1-\xi_k)|\cos\varphi|}
\label{rate2}
\end{equation}
for the inexact RQI with MINRES. The asymptotic factors in
(\ref{rate1})--(\ref{rate2}) do not affect the cubic convergence
rate itself but a bigger asymptotic factor affects reduction amount
at each outer iteration and more outer iterations may be needed.

Some other comments are in order. First,
the bigger $\beta$ is, the bigger the asymptotic convergence factors
are and meanwhile the considerably bigger the factor in
(\ref{rate2}) is than that in (\ref{rate1}) if $\xi_k$ is not near
zero. In this case, the cubic convergence of both RQI and the
inexact RQI may not be very visible; furthermore, RQI may need more
outer iterations, and the inexact RQI may need more outer iterations
than RQI. Second, if the factor in (\ref{rate1}) is not big, the
factor in (\ref{rate2}) differs not much with it
provided $\xi_k\leq\xi<1$ with a fixed $\xi$ not near one, so that
the inexact RQI and RQI use (almost) the same outer iterations.
Third, it is worth reminding that the factor in (\ref{rate2}) is an
estimate in the worst case, so it may often be conservative. Fourth,
as commented previously, the inexact RQI with MINRES may
behave more like quadratic in case $|\cos\varphi|$ is
occasionally very small or even zero.

Theorem~\ref{minres2} shows that $\|r_k\|$ can be used to control
$\xi_k$ in order to achieve a desired convergence rate. For cubic
convergence, at outer iteration $k$ we only need to solve the linear
system (\ref{EqERQILinearEquation1}) by MINRES with low accuracy
$\xi_k\leq\xi$. It is safe to do so with $\xi=0.1,0.5$ and even with
$\xi=0.8$. A smaller $\xi$ is not necessary and may cause much waste
at each outer iteration. Thus, we may save much computational cost,
compared with the inexact RQI with MINRES with decreasing tolerance
$\xi_k=O(\sin\phi_k)=O(\|r_k\|)$. Compared with Theorem~\ref{eshof},
another fundamental distinction is that the new quadratic
convergence results only require to solve the linear system with
very little accuracy $\xi_k=1-O(\sin\phi_k)=1-O(\|r_k\|)\approx 1$
rather than with $\xi_k\leq\xi$ not near one. They indicate that the
inexact RQI with MINRES converges quadratically provided that
$\cos\phi_k+\xi_k\cos\psi_k\approx
1-\xi_k=O(\sin\phi_k)=O(\|r_k\|)$. The results also illustrate that
the method converges linearly provided that
$\xi_k=1-O(\sin^2\phi_k)=1-O(\|r_k\|^2)$. In this case, we have
$\cos\phi_k+\xi_k\cos\psi_k\approx
1-\xi_k=O(\sin^2\phi_k)=O(\|r_k\|^2)$. So $\xi_k$ can be
increasingly closer to one as the method converges when quadratic
and linear convergence is required, and $\xi_k$ can be closer to one
for linear convergence than for quadratic convergence. These results
make it possible to design effective criteria on how to best control inner
tolerance $\xi_k$ in terms of the outer iteration accuracy $\|r_k\|$
to achieve a desired convergence rate. In addition, interestingly,
we comment that, instead of quadratic convergence as the authors of
\cite{mgs06} claimed, Table~1 in \cite{mgs06} actually showed the
same asymptotic cubic convergence of the inexact RQI with MINRES for
a fixed $\xi=0.1$ as that for decreasing tolerance
$\xi_k=O(\|r_k\|)$ in the sense of (\ref{rate1}) and (\ref{rate2}).

Below we estimate $\|w_{k+1}\|$ in (\ref{EqIRQILinearEquation1})
obtained by MINRES and establish a lower bound for it.
As a byproduct, we also present a simpler but weaker quadratic
convergence result. Note that the exact solution of
$(A-\theta_k I)w=u_k$ is
$w_{k+1}=(A-\theta_k I)^{-1}u_k$, which corresponds to $\xi_k=0$ in
(\ref{EqIRQIwk}). Therefore, with a reasonably good $u_k$,
from (\ref{EqIRQIwk}), (\ref{error1}) and (\ref{parlett}) we have
\begin{eqnarray*}
\|w_{k+1}\|&=&\frac{\cos\phi_k}{|\theta_k-\lambda|}+O(\sin\phi_k)\\
&\approx&\frac{1}{|\theta_k-\lambda|}=\|(A-\theta_kI)^{-1}\|\\
&=&O\left(\frac{1}{\sin^2\phi_k}\right)=O\left(\frac{1}{\|r_k\|^2}\right).
\end{eqnarray*}

\begin{theorem}\label{mimic}
Asymptotically we have
\begin{eqnarray}
\|w_{k+1}\|&\geq&\frac{(1-\xi_k)|\lambda_2-\lambda|}{4\beta\|r_k\|^2},\label{wk+1}\\
%\|r_{k+1}\|&\leq&\frac{\sqrt{1-\xi_k^2}}{\|w_{k+1}\|},\label{simoncini2}\\
\|r_{k+1}\|&\leq&\sqrt{\frac{1+\xi_k}{1-\xi_k}}\frac{4\beta}{|\lambda_2-\lambda|}
\|r_k\|^2. \label{resrelation}
\end{eqnarray}
Thus, the inexact RQI
with MINRES converges quadratically at least once $\xi_k$ is
not near one.
\end{theorem}

\begin{proof}

By using (\ref{EqIRQIwk}), (\ref{error1}), (\ref{parlett}) and
(\ref{denominres}) in turn, we obtain
\begin{eqnarray*}
\|w_{k+1}\|&\geq&\frac{|\cos\phi_k+\xi_k\cos\psi_k|}{|\theta_k-\lambda|}\\
&\geq&\frac{|\cos\phi_k+\xi_k\cos\psi_k|}{(\lambda_{\max}-\lambda_{\min})\sin^2\phi_k}\\
&\geq&\frac{|\cos\phi_k+\xi_k\cos\psi_k||\lambda_2-\lambda|}{4\beta\|r_k\|^2}\\
&\geq&\frac{(1-\xi_k)|\lambda_2-\lambda|+O(\sin^2\phi_k)}{4\beta\|r_k\|^2}\\
&=&\frac{(1-\xi_k)|\lambda_2-\lambda|}{4\beta\|r_k\|^2}
\end{eqnarray*}
with the last equality holding asymptotically by ignoring $O(1)$. This
proves (\ref{wk+1}).

It follows from  (\ref{minpro}) and $u_{k+1}=w_{k+1}/\|w_{k+1}\|$
that
$$
\|(A-\theta_k I)u_{k+1}\|=\frac{\sqrt{1-\xi_k^2}}{\|w_{k+1}\|}.
\label{simoncini1}
$$
So from the optimality of Rayleigh quotient we obtain
\begin{equation}
\|r_{k+1}\|=\|(A-\theta_{k+1}I)u_{k+1}\|\leq\|(A-\theta_kI)u_{k+1}\|
=\frac{\sqrt{1-\xi_k^2}}{\|w_{k+1}\|}.\label{simoncini2}
\end{equation}
Substituting (\ref{wk+1}) into it establishes (\ref{resrelation}).
\end{proof}

Simoncini and Eld$\acute{e}$n \cite{SimonciniRQI} present an
important estimate on $\|w_{k+1}\|$:
\begin{equation}
\|w_{k+1}\|\geq\frac{|1-\varepsilon_m|\cos^3\phi_k}{\sin\phi_k}
\frac{1}{\|r_k\|},\label{simonwk+1}
\end{equation}
where $\varepsilon_m=p_m(\lambda-\theta_k)$ with $p_m$ the residual
polynomial of MINRES satisfying $p_m(0)=1$; see Proposition 5.3 there.
This relation involves $\varepsilon_m$ and is less easily interpreted
than (\ref{wk+1}). When $\varepsilon_m$ is not near one,
$\|w_{k+1}\|$ is bounded by $O(\frac{1}{\|r_k\|^2})$ from below.
Based on this estimate, Simoncini and Eld$\acute{e}$n have designed
a stopping criterion for inner iterations.

From (\ref{wk+1}) and Theorems~\ref{minrescubic}--\ref{minres2}, it
is instructive to observe the remarkable facts: $\|w_{k+1}\|$
increases as rapidly as $O(\frac{1}{\|r_k\|^2})$ and
$O(\frac{1}{\|r_k\|})$, respectively, if the inexact RQI with MINRES
converges cubically and quadratically; but it is $O(1)$ if the
method converges linearly. As (\ref{wk+1}) is sharp, the size of
$\|w_{k+1}\|$ can reveal both cubic and quadratic convergence.
We can control $\|w_{k+1}\|$ to make the method converge
cubically or quadratically, as done in
\cite{SimonciniRQI}. However, it is unlikely to do so for linear
convergence as the method may converge linearly or disconverges when
$\|w_{k+1}\|$ remains $O(1)$.

Another main result of Simoncini and
Eld$\acute{e}$n~\cite{SimonciniRQI} is Proposition 5.3 there:
\begin{equation}
\|r_{k+1}\|\leq\frac{\sin\phi_k}{\cos^3\phi_k}\frac{\sqrt{1-\xi_k^2}}
{|1-\varepsilon_m|}\|r_k\|. \label{simoncini3}
\end{equation}
Note $\sin\phi_k=O(\|r_k\|)$. The above result means quadratic
convergence if $\frac{\sqrt{1-\xi_k^2}} {|1-\varepsilon_m|}$ is
moderate, which is the case if $\xi_k$ and $\varepsilon_m$ are not
near one. We refer to \cite{paige,xueelman} for discussions on
$\varepsilon_m$. Since how $\xi_k$ and $\varepsilon_m$ affect each
other is complicated, (\ref{resrelation}) is simpler and more easily
understandable than (\ref{simoncini3}). However, both
(\ref{resrelation}) and (\ref{simoncini3}) are weaker than
Theorems~\ref{minrescubic}--\ref{minres2} since quadratic
convergence requires $\xi_k<1$ not near one and the cubic
convergence of RQI and of the inexact RQI cannot be recovered when
$\xi_k=0$ and $\xi_k=O(\|r_k\|)$, respectively.

\section{The inexact RQI with a tuned preconditioned
MINRES}\label{precondit}

We have found that even for $\xi_k$ near one we may still need quite
many inner iteration steps at each outer iteration. This is especially
the case for difficult problems, i.e., big $\beta$'s, or for computing
an interior eigenvalue $\lambda$ since it
leads to a highly Hermitian indefinite matrix $(A-\theta_k I)$
at each outer iteration. So, in order to improve the
overall performance, preconditioning may be necessary to speed up
MINRES. Some preconditioning techniques have been proposed in, e.g.,
\cite{mgs06,SimonciniRQI}. In the unpreconditioned case, the
right-hand side $u_k$ of (\ref{EqERQILinearEquation1}) is rich in
the direction of the desired $x$. MINRES can benefit much from this
property when solving the linear system. Actually,
if the right-hand side is an eigenvector of the coefficient
matrix, Krylov subspace type methods will find the exact solution in one step.
However, a usual preconditioner loses this important
property, so that inner iteration steps may not be reduced
\cite{mgs06,freitagspence,freitag08b}. A preconditioner with tuning
can recover this property and meanwhile attempts to improve the
conditioning of the preconditioned linear system, so that considerable
improvement over a usual preconditioner is possible
\cite{freitagspence,freitag08b,xueelman}. In this section we show how to extend
our theory to the inexact RQI with a tuned preconditioned MINRES.

Let $Q=LL^*$ be a Cholesky factorization of some Hermitian positive
definite matrix which is an approximation to $A-\theta_k I$ in some
sense \cite{mgs06,freitag08b,xueelman}. A tuned preconditioner
${\cal Q}={\cal L}{\cal L}^*$ can be constructed by adding a rank-1
or rank-2 modification to $Q$, so that
\begin{equation}
{\cal Q}u_k=Au_k; \label{tune}
\end{equation}
see \cite{freitagspence,freitag08b,xueelman} for details. Using
the tuned preconditioner ${\cal Q}$, the shifted inner linear system
(\ref{EqERQILinearEquation1}) is equivalently transformed to
the preconditioned one
\begin{equation}
B\hat{w}={\cal L}^{-1}(A-\theta_k I){\cal L}^{-*}\hat{w}={\cal
L}^{-1}u_k \label{precond}
\end{equation}
with the original $w={\cal L}^{-*}\hat{w}$. Once MINRES is used to
solve it, we are led to the inexact RQI with a tuned preconditioned
MINRES. A power of the tuned preconditioner ${\cal Q}$  is that the
right-hand side ${\cal L}^{-1}u_k$ is rich in the eigenvector of $B$
associated with its smallest eigenvalue and has the same quality as
$u_k$ as an approximation to the eigenvector $x$ of $A$, while for the
usual preconditioner $Q$ the right-hand side $L^{-1}u_k$ does not
possess this property.

Take the zero vector as an initial guess to the solution of
(\ref{precond}) and let $\hat{w}_{k+1}$ be the approximate solution
obtained by the $m$-step MINRES applied to it. Then we have
\begin{equation}
{\cal L}^{-1}(A-\theta_k I){\cal L}^{-*}\hat{w}_{k+1}={\cal
L}^{-1}u_k+\hat{\xi}_k\hat{d}_k, \label{preresidual}
\end{equation}
where $\hat{w}_{k+1}\in {\cal K}_m(B,{\cal L}^{-1}u_k)$,
$\hat{\xi}_k\hat{d}_k$ with $\|\hat{d}_k\|=1$ is the residual and
$\hat{d}_k$ is the residual direction vector. Trivially, for any
$m$, we have $\hat{\xi}_k\leq\|{\cal L}^{-1}u_k\|$. Keep in mind
that $w_{k+1}={\cal L}^{-*}\hat{w}_{k+1}$. We then get
\begin{equation}
(A-\theta_k I)w_{k+1}=u_k+\hat{\xi}_k{\cal
L}\hat{d}_k=u_k+\hat{\xi}_k\|{\cal L}\hat{d}_k\|\frac{{\cal
L}\hat{d}_k}{\|{\cal L}\hat{d}_k\|}.\label{orig}
\end{equation}
So $\xi_k$ and $d_k$ in (\ref{EqIRQILinearEquation1}) are
$\hat{\xi}_k\|{\cal L}\hat{d}_k\|$ and $\frac{{\cal
L}\hat{d}_k}{\|{\cal L}\hat{d}_k\|}$, respectively, and our general
Theorems~\ref{ThmIRQIQuadraticConvergence}--\ref{resbound} apply and
is not repeated. In practice, we require that the tuned
preconditioned MINRES solves the inner linear system with
$\hat{\xi}_k<1/\|{\cal L}\hat{d}_k\|$ such that $\xi_k<1$.

How to extend Theorem~\ref{minres} to the preconditioned case is
nontrivial and needs some work. Let $(\mu_i,y_i),\ i=1,2,\ldots,n$
be the eigenpairs of $B$ with
$$
\mid\mu_1\mid<\mid\mu_2\mid\leq\cdots\leq\mid\mu_n\mid.
$$
Define $\hat{u}_k={\cal L}^{-1}u_k/\|{\cal L}^{-1}u_k\|$. Similar to
{\rm (\ref{EqIRQIDecompositoinOfu_k})} and {\rm
(\ref{EqIRQIDecompositoinOfd_k}), let
\begin{eqnarray}
\hat{u}_k&=&y_1\cos\hat{\phi}_k+\hat{e}_k\sin\hat{\phi}_k,\ \hat{e}_k\perp y_1,
\ \|\hat{e}_k\|=1, \label{predec1}\\
\hat{d}_k&=&y_1\cos\hat{\psi}_k+\hat{f}_k\sin\hat{\psi}_k, \
\hat{f}_k\perp y_1, \ \|\hat{f}_k\|=1 \label{predec2}
\end{eqnarray}
be the orthogonal direct sum decompositions of $\hat{u}_k$ and
$\hat{d}_k$. Then it is known
\cite{freitag08b} that
\begin{eqnarray}
|\mu_1|&=&O(\sin\phi_k),\label{mu1}\\
\sin\hat{\phi}_k&\leq& c_1\sin\phi_k \label{phihat}
\end{eqnarray}
with $c_1$ a constant.

Based on the proof line of Theorem~\ref{minres} and combining
(\ref{phihat}), the following results can be proved directly.

\begin{lemma}\label{minres4}
For the tuned preconditioned MINRES, let the unit length vectors $\hat{e}_k$ and
$\hat{f}_k$ be as in {\rm (\ref{predec1})} and {\rm
(\ref{predec2})}, define the angle
$\hat{\varphi}_k=\angle(\hat{f}_k, B\hat{e}_k)$ and assume that
\begin{equation}
|\cos\hat{\varphi}_k|\geq|\cos\hat{\varphi}|>0 \label{passump1}
\end{equation}
holds uniformly for an angle $\hat\varphi$ away from $\frac{\pi}{2}$
independent of $k$. Then we have
\begin{eqnarray}
 \sin\hat{\psi}_k&\leq&\frac{c_1|\mu_n|}{|\mu_2||\cos\hat\varphi|}
 \sin\phi_k,\label{psinpsi}\\
 \cos\hat{\psi}_k&=&\pm (1-O(\sin^2\phi_k)), \label{pcospsi}\\
\hat{d}_k&=&\pm y_1+O(\sin\phi_k).\label{pdkerror}
\end{eqnarray}
\end{lemma}
We can make a qualitative analysis, similar to that done for
Theorem~\ref{minres} in the unpreconditioned case, and justify the
assumption on $\hat{\varphi}_k$. With this theorem, we now estimate
$\sin\psi_k$ and $\cos\psi_k$.

\begin{theorem}\label{pminres}
For the tuned preconditioned MINRES, under the assumptions of
Theorem~\ref{minres2}, it holds that
\begin{eqnarray}
\sin\psi_k&=&O(\sin\phi_k),\\
\cos\psi_k&=&\pm (1-O(\sin^2\phi_k))\\
d_k&=&=\pm x+O(\sin\phi_k).
\end{eqnarray}
\end{theorem}

\begin{proof}
By definition, we have
$$
{\cal L}^{-1}(A-\theta_k I){\cal L}^{-*}y_1=\mu_1y_1,
$$
from which it follows that
$$
(A-\theta_kI)\tilde{u}_k=\mu_1\frac{{\cal L}y_1}{\|{\cal L}^{-*}y_1\|}
$$
with $\tilde{u}_k={\cal L}^{-*}y_1/\|{\cal L}^{-*}y_1\|$. Therefore,
by standard perturbation theory and (\ref{mu1}), we get
$$
\sin\angle(\tilde{u}_k,x)=O(\mu_1\frac{\|{\cal L}y_1\|}
{\|{\cal L}^{-*}y_1\|})=O(|\mu_1|)=O(\sin\phi_k).
$$
Since
$$
\angle(\tilde{u}_k,u_k)\leq\angle(\tilde{u}_k,x)+\angle(u_k,x),
$$
we get
$$
\sin\angle(\tilde{u}_k,u_k)=O(\sin\phi_k),
$$
i.e.,
$$
\tilde{u}_k=u_k+O(\sin\phi_k).
$$
As a result, from $Au_k={\cal Q}u_k$  and $\tilde{u}_k={\cal
L}^{-*}y_1/\|{\cal L}^{-*}y_1\|$ we have
\begin{eqnarray*}
{\cal L}y_1&=&\|{\cal L}^{-*}y_1\|{\cal L}{\cal L}^{*}\tilde{u}_k
=\|{\cal L}^{-*}y_1\|{\cal Q}\tilde{u}_k\\
&=&\|{\cal L}^{-*}y_1\|{\cal Q}
(u_k+O(\sin\phi_k))\\
&=&\|{\cal L}^{-*}y_1\|Au_k+O(\sin\phi_k)\\
&=&\|{\cal L}^{-*}y_1\|(\theta_ku_k+r_k)+O(\sin\phi_k)\\
&=&\theta_k\|{\cal L}^{-*}y_1\|u_k+\|{\cal L}^{-*}y_1\|r_k+O(\sin\phi_k),
\end{eqnarray*}
from which we obtain
$$
\sin\angle({\cal L}y_1,u_k)=O(\|r_k\|)+O(\sin\phi_k)=O(\sin\phi_k).
$$
Therefore, we have
\begin{equation}
\sin\angle({\cal L}y_1,x)\leq\sin\angle({\cal L}y_1,u_k)+
\sin\angle(u_k,x)=O(\sin\phi_k),\label{ly1}
\end{equation}
that is, the (unnormalized) ${\cal L}y_1$ has the same quality as
$u_k$ as an approximation to $x$. We have the orthogonal direct sum
decomposition
\begin{equation}
{\cal L}y_1=\|{\cal L}y_1\|\left(x\cos\angle({\cal
L}y_1,x)+g_k\sin\angle({\cal L}y_1,x)\right) \label{decly1}
\end{equation}
with $\|g_k\|=1$ and $g_k\perp x$. Also, make the orthogonal direct
orthogonal sum decomposition
\begin{equation}
{\cal L}\hat{f}_k=\|{\cal L}\hat{f}_k\|\left(x\cos\angle({\cal
L}\hat{f}_k,x)+h_k\sin\angle({\cal
L}\hat{f}_k,x)\right)\label{declfk}
\end{equation}
with $\|h_k\|=1$ and $h_k\perp x$. Making use of (\ref{psinpsi}) and
(\ref{pcospsi}), we get from (\ref{predec2}) that
\begin{eqnarray*}
\|{\cal L}_k\hat{d}_k\|^2&=&\|{\cal L}y_1\|^2\cos^2\hat{\psi}_k
+\|{\cal L}\hat{f}_k\|^2\sin^2\hat{\psi}_k
+2Re({\cal L}y_1,{\cal L}\hat{f}_k)\sin\hat{\psi}_k\cos\hat{\psi}_k\\
&=&\|{\cal L}y_1\|^2+O(\sin^2\phi_k)+O(\sin\phi_k)\\
&=&\|{\cal L}y_1\|^2+O(\sin\phi_k).
\end{eqnarray*}
Hence we have
\begin{equation}
\|{\cal L}y_1\|=\|{\cal L}\hat{d}_k\|+O(\sin\phi_k).\label{length}
\end{equation}

Now, from (\ref{predec2}), we obtain
$$
d_k=\frac{{\cal L}\hat{d}_k}{\|{\cal L}\hat{d}_k\|}
=\frac{1}{\|{\cal L}\hat{d}_k\|}({\cal L}y_1\cos\hat{\psi}_k +{\cal
L}\hat{f}_k\sin\hat{\psi}_k).
$$
Substituting (\ref{decly1}) and
(\ref{declfk}) into the above relation yields
\begin{eqnarray*}
d_k&=&\frac{1}{\|{\cal L}\hat{d}_k\|}((\|{\cal
L}y_1\|\cos\angle({\cal L}y_1,x)\cos\hat{\psi}_k+\|{\cal
L}\hat{f}_k\|\cos\angle({\cal
L}\hat{f}_k,x)\sin\hat{\psi}_k)x\\
&&+(\|{\cal L}y_1\|g_k\sin\angle({\cal
L}y_1,x)\cos\hat{\psi}_k+\|{\cal L}\hat{f}_k\|h_k\sin\angle({\cal
L}\hat{f}_k,x)\sin\hat{\psi}_k)).
\end{eqnarray*}
This is the orthogonal direct sum decomposition of $d_k$.
Making use of $\sin\angle({\cal L}y_1,x)=O(\sin\phi_k)$,
$\sin\hat{\psi}_k=O(\sin\phi_k)$ and (\ref{length}) and comparing with
$d_k=x\cos\psi_k+e_k\sin\psi_k$, we get
$$
\sin\psi_k=O(\sin\phi_k).
$$
Thus, it holds that $\cos\psi_k=\pm (1-O(\sin^2\phi_k))$ and
$d_k=\pm x+O(\sin\phi_k)$.
\end{proof}

Compared with Theorem~\ref{minres}, this theorem indicates that the
directions of residuals
obtained by the unpreconditioned MINRES and the tuned preconditioned
MINRES have the same properties. With the theorem, we can
write $\sin\psi_k\leq c_2\sin\phi_k$ with $c_2$ a constant. Then, based on
Theorems~\ref{ThmIRQIQuadraticConvergence}--\ref{resbound}, it is
direct to extend Theorems~\ref{minrescubic}--\ref{minres2} to the
inexact RQI with the tuned preconditioned MINRES, respectively. We
have made preliminary experiments and confirmed the theory. Since
our main concerns in this paper are the convergence theory and a
pursue of effective tuned preconditioners are beyond the scope of
the current paper, we will only report numerical experiments on the
inexact RQI with the unpreconditioned MINRES in the next section.

\section{Numerical experiments}\label{testminres}

Throughout the paper, we perform numerical experiments on an Intel
(R) Core (TM)2 Quad CPU Q9400 $2.66$GHz with main memory 2 GB using
Matlab 7.8.0 with the machine precision $\epsilon_{\rm
mach}=2.22\times 10^{-16}$ under the Microsoft Windows XP operating
system.

We report numerical experiments on four symmetric (Hermitian)
matrices: BCSPWR08 of order 1624, CAN1054 of order 1054, DWT2680 of
order 3025 and LSHP3466 of order 3466 \cite{duff}. Note that the
bigger
$\beta=\frac{\lambda_{\max}-\lambda_{\min}}{|\lambda_2-\lambda|}$
is, the worse conditioned $x$ is. For a bigger $\beta$,
Theorem~\ref{ThmIRQIQuadraticConvergence} and
Theorems~\ref{minrescubic}--\ref{minres2} show that RQI and the
inexact RQI with MINRES may converge more slowly and use more outer
iterations though they can still converge cubically.
If an interior eigenpair is required, the shifted linear systems can
be highly indefinite (i.e., many positive and negative eigenvalues)
and may be hard to solve. As a reference, we use the Matlab function
{\sf eig.m} to compute $\beta$. To better illustrate the theory, we
compute both exterior and interior eigenpairs. We compute the
smallest eigenpair of BCSPWR08, the tenth smallest eigenpair of
CAN1054, the largest eigenpair of DWT2680 and the twentieth smallest
eigenpair of LSHP3466, respectively.

Remembering that the (asymptotic) cubic convergence of the inexact RQI for
$\xi_k=O(\|r_k\|)$ is independent of iterative solvers,
in the experiments we take
\begin{equation}
\xi_k=\min\{0.1,\frac{\|r_k\|}{\|A\|_1}\}.\label{decrease}
\end{equation}

We first test the inexact RQI with MINRES for $\xi_k\leq\xi<1$ with
a few constants $\xi$ not near one and illustrate its cubic
convergence. We construct the same initial $u_0$ for each matrix
that is $x$ plus a reasonably small perturbation generated randomly
in a uniform distribution, such that
$|\lambda-\theta_0|<\frac{|\lambda-\lambda_2|}{2}$. The algorithm
stops whenever $\|r_k\|=\|(A-\theta_k I)u_k\|\leq\|A\|_1tol$, and we
take $tol=10^{-14}$ unless stated otherwise. In experiments, we use
the Matlab function {\sf minres.m} to solve the inner linear
systems. Tables~\ref{BCSPWR08minres}--\ref{lshpminres} list the
computed results, where $iters$ denotes the number of total inner
iteration steps and $iter^{(k-1)}$ is the number of inner iteration
steps when computing $(\theta_k,u_k)$, the "-" indicates that MINRES
stagnates and stops at the $iter^{(k-1)}$-th step, and $res^{(k-1)}$
is the actual relative residual norm of the inner linear system when
computing $(\theta_k,u_k)$. Clearly, $iters$ is a reasonable measure
of the overall efficiency of the inexact RQI with MINRES. We comment
that in {\sf minres.m} the output $iter^{(k-1)}=m-1$, where $m$ is
the steps of the Lanczos process.

\begin{table}[ht]
\begin{center}
\begin{tabular}{|c|c|c|c|c|c|c|}\hline
$\xi_{k-1}\leq\xi$&$k$&$\|r_k\|$&$\sin\phi_k$&$res^{(k-1)}$
&$iter^{(k-1)}$&$iters$\\\hline
0 (RQI)&1& 0.0092&0.0025& & &  \\
&  2&$4.4e-8$&$5.0e-8$& & & \\
& 3&$1.0e-15$&$2.2e-15$ & & &\\\hline
$\frac{\|r_{k-1}\|}{\|A\|_1}$&1 &0.0096&0.0036&0.0423&6&126\\
&2&$8.4e-8$&$1.3e-7$&$5.5e-4$&37&\\
&3&$9.4e-15$&$3.2e-15$&-&83&\\\hline
0.1&1&0.0105&0.0049&0.0707&5&68\\
&2&$2.9e-6$&$2.4e-6$&0.0784&21&\\
&3&$1.3e-13$&$2.7e-13$&0.0863&42&\\\hline
0.5&1&0.0218&0.0111&0.2503&3&88\\
&2&$8.7e-5$&$1.9e-4$&0.4190&11&\\
&3&$6.3e-9$&$2.1e-8$&0.4280&31&\\
&4&$1.1e-14$&$3.3e-15$&-&43&\\\hline
$1-\frac{c_1\|r_{k-1}\|}{\|A\|_1}$
&1&0.1409&0.0363&0.8824&1&75\\
$tol=10^{-13}$&2&0.0068&0.2274&0.9227&3&\\
&3&$1.3e-4$&$5.4e-4$&0.9284&11&\\
&4&$3.5e-7$&$1.1e-6$&0.9845&19&\\
&5&$3.4e-11$&$1.6e-11$&$1-3.7\times 10^{-5}$&30&\\
&6&$1.3e-13$&$4.6e-13$&$1-2.2\times 10^{-8}$&11&\\
\hline
\end{tabular}
\begin{tabular}{|c|c|c|}\hline
$\xi_{k-1}$& $k \ (iter^{(k-1)})$& $iters$\\\hline
$1-\left(\frac{c_2\|r_{k-1}\|}{\|A\|_1}\right)^2$& &\\
$tol=10^{-10}$& 1 (1); 2 (3); 3 (9); 4 (13); 5 (15); 6 (16) &62
\\\hline
\end{tabular}
\caption{BCSPWR08, $\beta=40.19,\ \sin\phi_0=0.1134$,
$c_1=c_2=1000$. $k \ (iter^{(k-1)})$ denotes the number of inner
iteration steps used by MINRES when computing
$(\theta_k,u_k)$.}\label{BCSPWR08minres}
\end{center}
\end{table}

\begin{table}[ht]
\begin{center}
\begin{tabular}{|c|c|c|c|c|c|c|}\hline
$\xi_{k-1}\leq\xi$&$k$&$\|r_k\|$&$\sin\phi_k$&$res^{(k-1)}$&
$iter^{(k-1)}$&$iters$\\\hline
0 (RQI)&1&0.0139&0.038& & &  \\
&  2&$6.1e-6$&$5.0e-5$& & & \\
& 3&$1.4e-14$&$3.6e-13$& & &\\\hline
$\frac{\|r_{k-1}\|}{\|A\|_1}$&1 &0.0196&0.0110&0.0387&12&606\\
&2&$6.7e-7$&$1.5e-5$&$3.5e-4$&184&\\
&3&$8.7e-14$&$4.7e-13$&-&410&\\\hline
0.1&1&0.0231&0.0155&0.0943&6&394\\
&2&$2.8e-7$&$8.4e-7$&0.0816&178&\\
&3&$7.6e-14$&$4.7e-13$&-&210&\\\hline
0.5&1&0.0646&0.0253&0.3715&3&408\\
&2&$6.0e-4$&$0.0071$&0.4757&37&\\
&3&$1.2e-7$&$1.6e-7$&0.4636&165&\\
&4&$4.5e-14$&$4.8e-13$&-&203&\\\hline
$1-\frac{c_1\|r_{k-1}\|}{\|A\|_1}$
&1&0.2331&0.0541&0.8325&1&536\\
$tol=10^{-12}$&2&0.0202&0.0161&0.90551&4&\\
&3&$2.5e-4$&0.0044&0.9469&64&\\
&4&$2.1e-6$&$3.2e-5$&0.9913&149&\\
&5&$3.1e-9$&$1.8e-8$&0.9999&155&\\
&6&$2.5e-12$&$2.8e-11$&$1-1.6\times 10^{-7}$&163&\\\hline
\end{tabular}
\begin{tabular}{|c|c|c|}\hline
$\xi_{k-1}$& $k \ (iter^{(k-1)})$& $iters$\\\hline
$1-\left(\frac{c_2\|r_{k-1}\|}{\|A\|_1}\right)^2$& &\\
$tol=10^{-10}$& 1 (1); 2 (3); 3 (16); 4 (172); 5 (143) &335
\\\hline
\end{tabular}
\caption{CAN1054, $\beta=88.28,\ \sin\phi_0=0.1137$, $c_1=c_2=1000$.
$k \ (iter^{(k-1)})$ denotes the number of inner iteration steps
used by MINRES when computing
$(\theta_k,u_k)$.}\label{can1054minres}
\end{center}
\end{table}

\begin{table}[ht]
\begin{center}
\begin{tabular}{|c|c|c|c|c|c|c|}\hline
$\xi_{k-1}\leq\xi$&$k$&$\|r_k\|$&$\sin\phi_k$&$res^{(k-1)}$
&$iter^{(k-1)}$ &$iters$\\\hline
0 (RQI)&1&0.0084&0.1493& & &  \\
&  2&$1.6e-4$&0.0047& & & \\
& 3&$3.2e-9$&$1.1e-7$ & & &\\
&4&$2.0e-15$&$1.1e-13$ & & &\\\hline
$\frac{\|r_{k-1}\|}{\|A\|_1}$&1 &0.0048&0.0409&0.0755&13&595\\
&2&$3.1e-6$&$1.1e-4$&$6.0e-4$&115&\\
&3&$7.6e-14$&$1.3e-12$&-&192&\\
&4&$1.1e-14$&$2.1e-13$&-&275&\\\hline
0.1&1&0.0054&0.0492&0.0966&10&313\\
&2&$1.3e-5$&$2.9e-4$&0.0959&50&\\
&3&$2.7e-10$&$6.8e-9$&0.0967&114&\\
&4&$7.0ee-13$&$2.2e-13$&-&139&\\ \hline
0.5&1&0.0228&0.0705&0.4717&3&297\\
&2&$3.9e-4$&$0.0136$&0.4918&28&\\
&3&$2.6e-6$&$1.3e-4$&0.4681&42&\\
&4&$1.2e-10$&$2.9e-9$&0.4579&109&\\
&5&$6.2e-14$&$3.3e-13$&-&115&\\\hline
$1-\frac{c_1\|r_{k-1}\|}{\|A\|_1}$
&1&0.1031&0.0800&0.9309&1&244\\
&2&0.0081&0.0617&0.9369&5&\\
&3&$8.0e-4$&$0.0025$&0.9470&17&\\
&4&$3.7e-6$&$5.1e-4$&0.9878&31&\\
&5&$3.6e-8$&$7.5e-7$&$1-5.6\times 10^{-4}$&77&\\
&6&$1.4e-12$&$4.5e-11$&$1-5.9\times 10^{-5}$&113&\\\hline
\end{tabular}
\begin{tabular}{|c|c|c|}\hline
$\xi_{k-1}$& $k \ (iter^{(k-1)})$& $iters$\\\hline
$1-\left(\frac{c_2\|r_{k-1}\|}{\|A\|_1}\right)^2$&1 (1); 2 (3); 3
(9);
4 (18);5 (22) &\\
$tol=10^{-9}$&  6 (31); 7 (34); 8 (25); 9 (26); 10 (26) &195
\\\hline
\end{tabular}
\caption{DWT2680, $tol=10^{-12}$, $\beta=2295.6,\
\sin\phi_0=0.1133$, $c_1=10000,\ c_2=1000$. $k \ (iter^{(k-1)})$
denotes the number of inner iteration steps used by MINRES when
computing $(\theta_k,u_k)$.}\label{dwtminres}
\end{center}
\end{table}

\begin{table}[ht]
\begin{center}
\begin{tabular}{|c|c|c|c|c|c|c|}\hline
$\xi_{k-1}\leq\xi$&$k$&$\|r_k\|$&$\sin\phi_k$&$res^{(k-1)}$&
$iter^{(k-1)}$&$iters$\\\hline
0 (RQI)&1&0.0111&0.1096& & &  \\
&  2&$1.0e-4$&0.0016& & & \\
& 3&$2.3e-10$&$9.0e-8$ & & &\\
&4&$2.0e-15$&$6.5e-13$ & & &\\\hline
$\frac{\|r_{k-1}\|}{\|A\|_1}$&1 &0.0100&0.0123&0.0998&5&2692\\
&2&$8.2e-7$&$8.9e-5$&$1.4e-3$&697&\\
&3&$1.3e-13$&$9.1e-13$&-&1110&\\
&4&$4.2e-14$&$6.5e-13$&-&880&\\\hline
0.1&1&0.0100&0.0077&0.09898&5&1902\\
&2&$6.3e-6$&0.0018&0.0996&223&\\
&3&$3.6e-10$&$2.9e-8$&0.0975&790&\\
&4&$1.9e-13$&$6.5e-13$&-&862&\\ \hline
0.5&1&0.0353&0.0270&0.4302&2&1710\\
&2&$3.8e-4$&$0.0064$&0.4838&14&\\
&3&$3.6e-7$&$2.5e-4$&0.4938&119&\\
&4&$3.1e-11$&$1.8e-9$&0.4794&767&\\
&5&$1.5e-13$&$6.5e-13$&-&808&\\\hline
$1-\frac{c_1\|r_{k-1}\|}{\|A\|_1}$
&1&0.0795&0.0369&0.7382&1&1967\\
$tol=10^{-12}$&2&0.0045&0.0117&0.9037&5&\\
&3&$9.2e-5$&$0.0039$&0.9454&35&\\
&4&$3.3e-7$&$0.0002$&0.9863&611&\\
&5&$4.9e-9$&$2.1e-6$&$1-4.9\times 10^{-5}$&627&\\
&6&$5.6e-12$&$1.0e-9$&$1-7.8\times 10^{-7}$&688&\\ \hline
\end{tabular}
\begin{tabular}{|c|c|c|}\hline
$\xi_{k-1}$& $k \ (iter^{(k-1)})$& $iters$\\\hline
$1-\left(\frac{c_2\|r_{k-1}\|}{\|A\|_1}\right)^2$&1 (1); 2 (3); 3
(8);
4 (84);5 (631) &\\
$tol=10^{-10}$&  6 (655); 7 (203); 8 (496) &2081
\\\hline
\end{tabular}
\caption{LSHP3466, $tol=10^{-13}$, $\beta=2613.1,\
\sin\phi_0=0.1011$, $c_1=c_2=1000$. $k \ (iter^{(k-1)})$ denotes the
number of inner iteration steps used by MINRES when computing
$(\theta_k,u_k)$.}\label{lshpminres}
\end{center}
\end{table}

Before explaining our experiments, we should remind that in finite
precision arithmetic $\|r_k\|/\|A\|_1$ cannot decrease further
whenever it reaches a moderate multiple of $\epsilon_{\rm
mach}=2.2\times 10^{-16}$. Therefore, assuming that the algorithm
stops at outer iteration $k$, if $\sin\phi_{k-1}$ or $\|r_{k-1}\|$
is at the level of $10^{-6}$ or $10^{-9}$, then the algorithm may
not continue converging cubically or quadratically at the final
outer iteration $k$.

To judge cubic convergence, we again stress that we should rely on
(\ref{rate1}) and (\ref{rate2}) for RQI and the inexact RQI with
MINRES, respectively. We observe from the tables that the inexact RQI
with MINRES for $\xi_k\leq\xi<1$ with $\xi$ fixed not near one converges
cubically and behaves like RQI and the inexact RQI with MINRES with
decreasing tolerance $\xi_k=(\|r_k\|)$; it uses (almost) the same
outer iterations as the latter two do. The results clearly indicate
that cubic convergence is generally insensitive to $\xi$ provided $\xi$
is not near one. Furthermore, we see that the algorithm with a fixed $\xi$
not near one is much more efficient than the algorithm with
$\xi_k=O(\|r_k\|)$ and is generally about one and a half to twice as
fast as the latter. Since $(A-\theta_k I)w=u_k$ becomes increasingly
ill conditioned as $k$ increases, we need more inner iteration steps
to solve the inner linear system with the same accuracy $\xi$,
though the right-hand side $u_k$ is richer in the direction of $x$
as $k$ increases. For $\xi_k=O(\|r_k\|)$, inner iteration steps
needed can be much more than those for a fixed $\xi$ not near one at
each outer iteration as $k$ increases. We refer to \cite{mgs06} for
a descriptive analysis.

For the above numerical tests, we pay special attention to the ill
conditioned DWT2680 and LSHP3466. Intuitively, RQI and the inexact
RQI with MINRES seems to exhibit quadratic convergence. However, it
indeed converges cubically in the sense of (\ref{rate1}) and
(\ref{rate2}). With $\xi=0.5$, $\sin\phi_k$ and $\|r_k\|$ decrease
more slowly than those obtained with $\xi=0.1$ and the exact RQI as
well as the inexact RQI with decreasing tolerance, and the algorithm
uses one more outer iteration. This is because the convergence
factors in both (\ref{rate1}) and (\ref{rate2}) are big and the
factor with $\xi=0.5$ is considerably bigger than those with the
others. However, the method with $\xi=0.5$ uses comparable $iters$.

Our experiments show that the inexact RQI with MINRES is not
sensitive to $\xi<1$ not near one. So it is advantageous to
implement the inexact RQI with MINRES with a fixed $\xi$ not near
one so as to achieve the cubic convergence. We can benefit much from
such a new implementation and use possibly much fewer $iters$,
compared with the method with $\xi_k=O(\|r_k\|)$.

Next we confirm Theorems~\ref{minrescubic}--\ref{minres2} and verify
quadratic convergence and linear convergence when conditions
(\ref{quadracond}) and (\ref{linrescond}) are satisfied,
respectively. Note that $\beta$ and $|\cos\varphi|$ in the upper
bounds for $\xi_k$ are uncomputable a priori during the process.
However, by their forms we can take
\begin{equation}
\xi_k=1-\frac{c_1\|r_k\|}{\|A\|_1} \label{critera}
\end{equation}
and
\begin{equation}
\xi_k=1-\left(\frac{c_2\|r_k\|}{\|A\|_1}\right)^2 \label{criterb}
\end{equation}
for reasonable $c_1$ and $c_2$, respectively, and use them to test
if the inexact RQI with MINRES converges quadratically and linearly.
It is seen from (\ref{quadracond}) and (\ref{linrescond}) that we
should take $c_1$ and $c_2$ bigger than one as $\beta\geq 1$,
$|\cos\varphi|\leq 1$ and $\zeta<1$. The bigger $\beta$ is, the
bigger $c_1$ and $c_2$ should be. Note that $\xi_k$ defined so may
be negative in the very beginning of outer iterations if $u_0$ is
not good enough. In our implementations, we take
\begin{equation}
\xi_k=\max\{0.95,1-\frac{c_1\|r_k\|}{\|A_1||}\} \label{criteria1}
\end{equation}
and
\begin{equation}
\xi_k=\max\{0.95,1-\left(\frac{c_2\|r_k\|}{\|A_1\|}\right)^2\}
\label{criteria2}
\end{equation}
with $100\leq c_1,c_2\leq 10000$ for quadratic and linear
convergence, respectively. As remarked previously, the inexact RQI
with MINRES for $\xi=0.8$ generally converges cubically though it
may reduce $\|r_k\|$ and $\sin\phi_k$ not as much as that for $\xi$
smaller at each outer iteration. We take it as a reference for cubic
convergence. We implement the method using (\ref{critera}) and
(\ref{criterb}), respectively, after very few outer iterations as
long as the algorithm starts converging. They must approach one as
outer iterations proceed. Again, we test the above four matrices. In
the experiments, we have taken several $c_1,c_2$'s ranging from 100
to 10000. The bigger $c_1$ and $c_2$ are, the safer are bounds
(\ref{criteria1}) and (\ref{criteria2}) for quadratic and linear
convergence, and the faster the algorithm converges. We report the
numerical results for $c_1=c_2=1000$ in
Tables~\ref{BCSPWR08minres}--\ref{lshpminres} except $c_1=10000$ for
DWT2680. Figure~\ref{fig1} draws the convergence curves of the
inexact RQI with MINRES for the four matrices for the fixed
$\xi_k=0.8$ and $c_1=c_2=1000$ except $c_1=10000$ for DWT2680.

\begin{figure}[ht]
\begin{center}
\includegraphics[width=7cm]{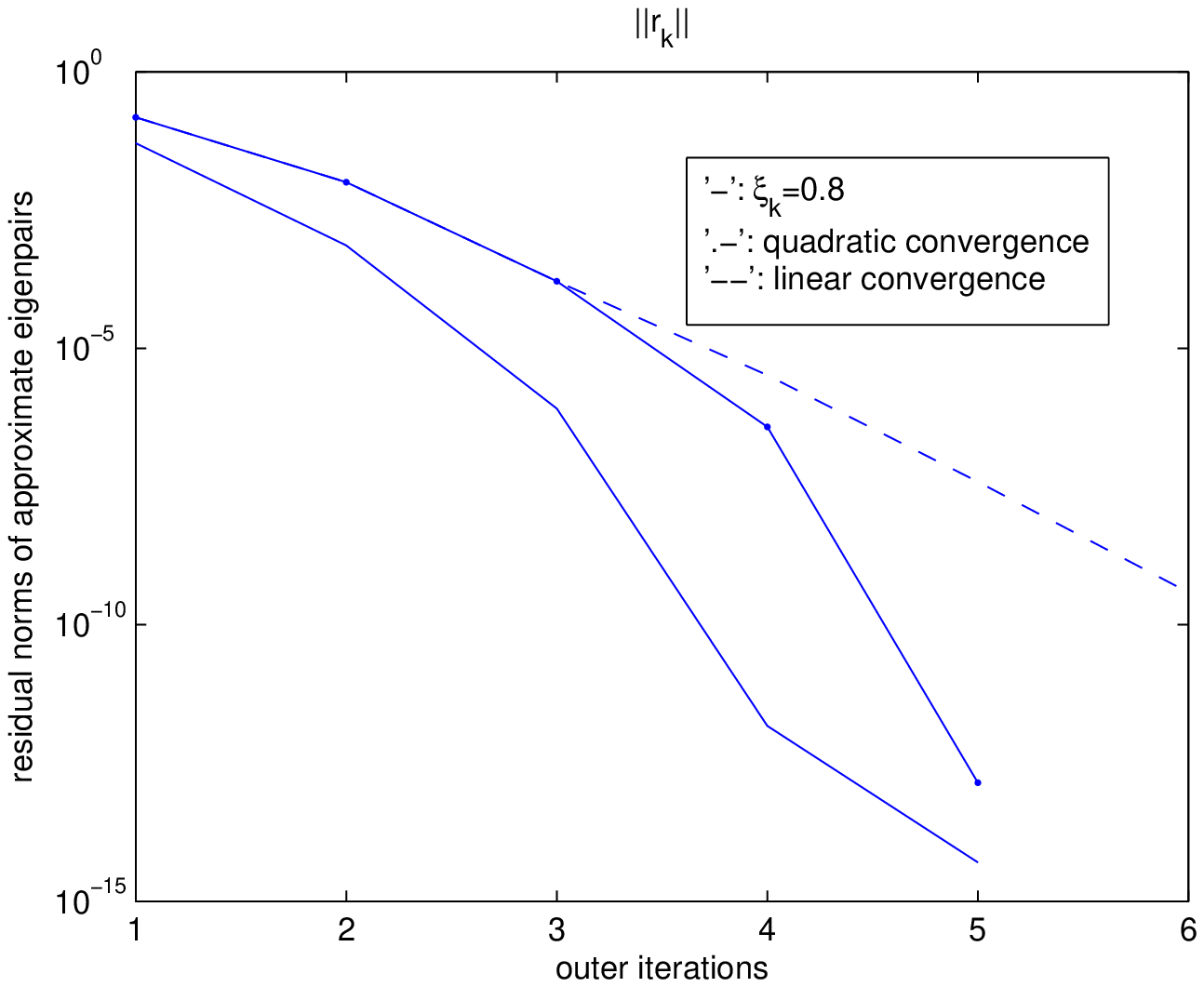}
\includegraphics[width=7cm]{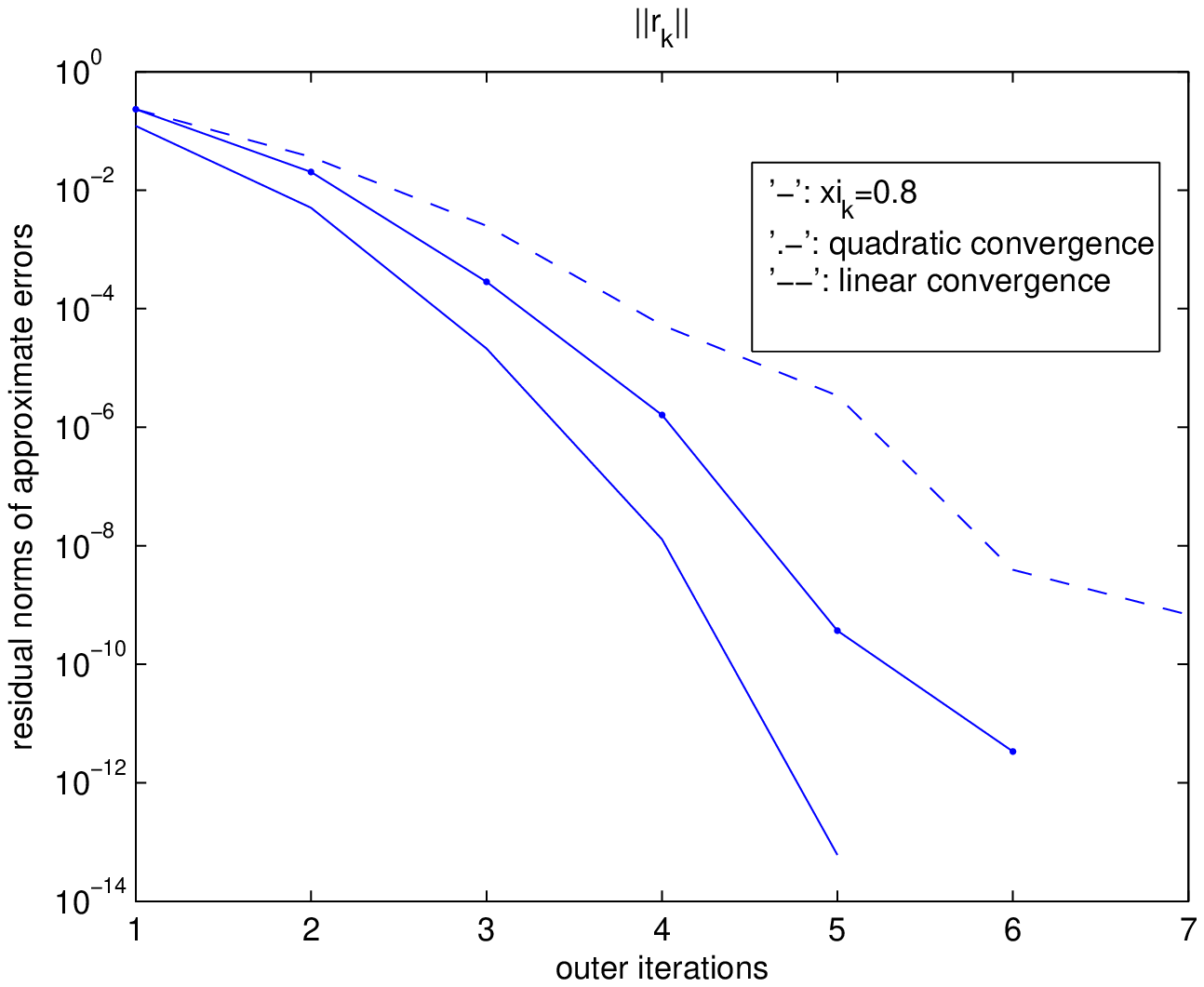}
\includegraphics[width=7cm]{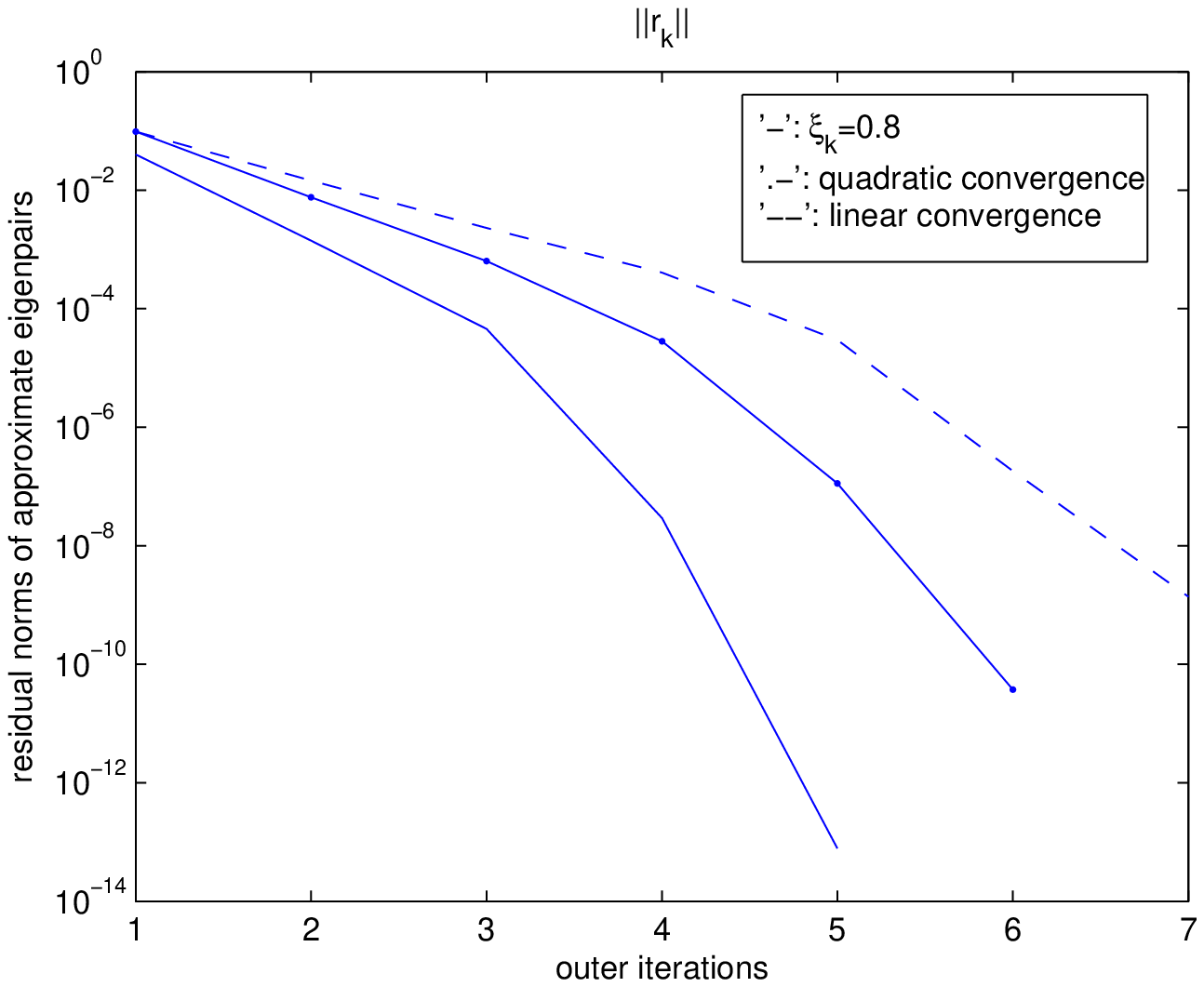}
\includegraphics[width=7cm]{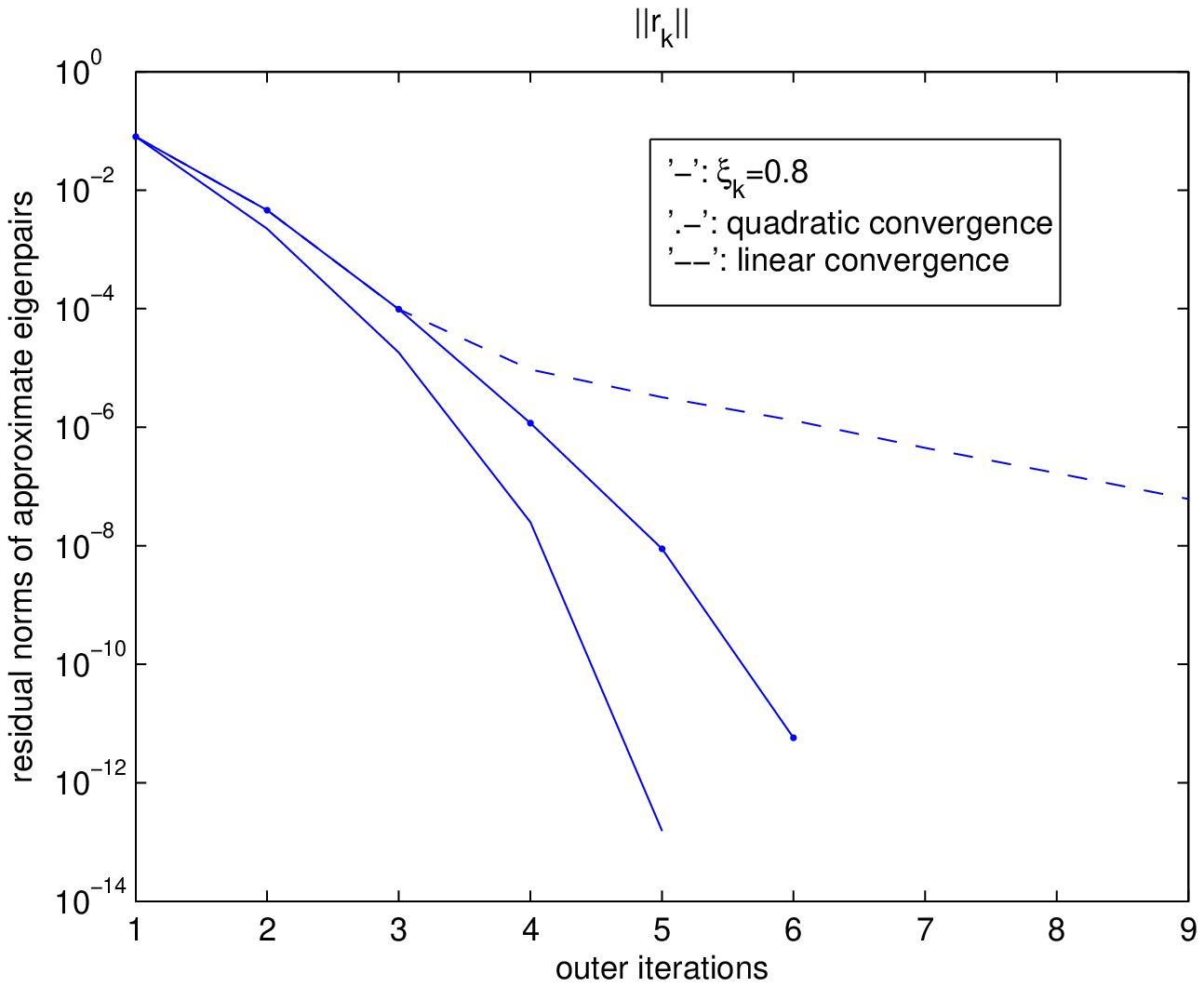}
\end{center}
\caption{Quadratic and linear convergence of the inexact RQI with
MINRES for BCSPWR08, CAN1054, DWT2680 and LSHP3466 in order, in
which the solid line denotes the convergence curve of
$\xi_k=\xi=0.8$, the dotted dash line the quadratic convergence
curve and the dashed line the linear convergence curve.}\label{fig1}
\end{figure}

Figure~\ref{fig1} clearly exhibits the typical behavior of
quadratic and linear convergence of the inexact RQI with MINRES.
Precise data details can be found in
Tables~\ref{BCSPWR08minres}--\ref{lshpminres}. As outer iterations
proceed, $\xi_k$ is increasingly closer to one but the algorithm
steadily converges quadratically and linearly; see the tables for
quadratic convergence. The tables and figure indicate that
our conditions (\ref{criteria1}) and (\ref{criteria2}) indicate the
inexact RQI works very well for chosen $c_1$ and $c_2$. For
quadratic convergence, $\xi_k$ becomes increasingly closer to one,
but on the one hand $iter^{(k-1)}$ still increases as outer
iterations proceed and on the other hand it is considerably smaller
than that with a fixed $\xi$. In contrast, for linear convergence,
$iter^{(k-1)}$ varies not much with increasing $k$ except for the
first two outer iterations, where $iter^{(k-1)}$ is no more than
five.

For other $c_1$ and $c_2$, we have made experiments in the same way.
We have observed similar phenomena for quadratic convergence and found
that the algorithm is not sensitive to $c_1$ in general, but this is not the
case for $c_2$. For different $c_2$, the method still converges
linearly but the number of outer iterations may vary quite a lot.
This should be expected as $c_2$ critically affects the linear
convergence factor $\zeta$ that uniquely determines convergence
speed, while $c_1$ does not affect quadratic convergence rate and
only changes the factor $\eta$ in the quadratic convergence bounds
(\ref{quadminres}) and (\ref{quadres}). Also, we should be careful
when using (\ref{criteria2}) in finite precision arithmetic. If
$$
\left(\frac{c_2\|r_k\|}{\|A\|_1}\right)^2\
$$
is at the level of $\epsilon$ or smaller for some $k$, then
(\ref{criteria2}) gives $\xi_k=1$ in finite precision arithmetic.
The inexact RQI with MINRES will break down and cannot continue the
$(k+1)$-th outer iteration. A adaptive strategy is to fix $\xi_k$ to
be a constant smaller than one once $\|r_k\|$ is so small that
$\xi_k=1$ in finite precision arithmetic. We found that
$\xi_k=1-10^{-8}$ is a reasonable choice. We have tested this
strategy for the four matrices and found that it works well.

\section{Concluding remarks}\label{conc}

We have considered the convergence of the inexact RQI without and
with MINRES in detail and have established a number of results on
cubic, quadratic and linear convergence. These results clearly show
how inner tolerance affects the convergence of outer iterations and
provide practical criteria on how to best control inner tolerance to
achieve a desired convergence rate. It is the first time to appear
surprisingly that the inexact RQI with MINRES generally converges
cubically for $\xi_k\leq\xi<1$ with $\xi$ a
constant not near one and quadratically
for $\xi_k$ increasingly near one, respectively.
They are fundamentally different from the existing
results and have a strong impact on effectively
implementing the algorithm so as to reduce the total
computational cost very considerably.

Using the same analysis approach in this paper, we have considered
the convergence of the inexact RQI with the unpreconditioned and
preconditioned Lanczos methods for solving inner linear systems
\cite{jia09}, where quadratic and linear convergence remarkably
allows $\xi_k\geq 1$ considerably, that is, approximate
solutions of the inner linear systems have no accuracy at all in the
sense of solving linear systems. By comparisons, we find that the
inexact RQI with MINRES is preferable in robustness and efficiency.

Although we have restricted to the Hermitian case, the analysis
approach could be used to study the convergence on the inexact
RQI with Arnoldi and GMRES for the non-Hermitian eigenvalue problem.

We have only considered the standard Hermitian eigenvalue problem
$Ax=\lambda x$ in this paper. For the Hermitian definite generalized
eigenvalue problem $Ax=\lambda Mx$ with $A$ Hermitian and $M$
Hermitian positive definite, if the $M$-inner product, the $M$-norm
and the $M^{-1}$-norm, the angle induced from the $M$-inner product
are properly placed in positions of the usual Euclidean inner
product, the Euclidean norm and the usual angle, then based on the
underlying $M$-orthogonality of eigenvectors of the matrix pair
$(A,M)$, we should be able to extend our theory developed in the paper to
the inexact RQI with the unpreconditioned and tuned preconditioned
MINRES for the generalized eigenproblem. This work is in progress.


\begin{thebibliography}{99}

\bibitem{mgs06} J. Berns-M\"{u}ller, I. G. Graham and A. Spence,
Inexact inverse iteration for symmetric matrices, {\em Linear
Algebra Appl.}, 416 (2006), pp.~389--413.

\bibitem{MullerVariableShift}
J. Berns-M\"{u}ller and A. Spence, Inexact inverse iteration with
variable shift for nonsymmetric generalized eigenvalue problems,
{\em SIAM J. Matrix Anal. Appl.}, 28 (2006), pp.~1069--1082.

%\bibitem{mullerspence} J. Berns-M\"{u}ller and A. Spence, Inexact inverse
%iteration and GMRES. Technical Report, 2007, Department of
%Mathematical Sciences, Unversity of Bath.

\bibitem{duff} I. S. Duff, R. G. Grimes and J. G. Lewis,
User's Guide for the Harwell-Boebing sparse matrix collection
(Release 1), Tech. Rep., RAL-92-086, Rutherford Appleton
Laboratory, UK, 1992. Data available at
http://math.nist.gov/MarketMatrix.


\bibitem{freitagspence} M. A. Freitag and A. Spence,
Convergence of inexact inverse iteration with application to
preconditioned iterative solves, {\em BIT}, (47) (2007), pp.~27--44.

%\bibitem{freitag07} M. A. Freitag and A. Spence, Convergence theory
%for inexact inverse iteration applied to the generalized
%nonsymmetric eigenproblem, {\em Elect. Trans. Numer. Anal}., 28
%(2007), pp.~40--67.

%\bibitem{freitag08}M. A. Freitag and A. Spence, Rayleigh quotient
%iteration and simplified Jacobi-Davidson method with preconditioned
%iterative solves, {\em Linear Algebra Appl.}, 428 (2008),
%pp.~2049--2060.

\bibitem{freitag08b} M. A. Freitag and A. Spence, A tuned
preconditioner for inexact inverse iteration applied to Hermitian
eigenvalue problems, {\em IMA J. Numer. Anal.}, 28 (2008),
pp.~522--551.

\bibitem{GolubMC}
G. H. Golub and C. F.~van Loan, {\em Matrix Computations}, The John
Hopkins University Press, Baltimore, London, 1996.

\bibitem{golubye} G. H. Golub and Q. Ye, Inexact inverse
iterations for generalized eigenvalue problems, {\em BIT}, 40
(2000), pp.~671--684.

%\bibitem{golubzhangzha} G. H. Golub, Z. Zhang and H. Zha, Large Sparse
%Symmetric Eigenvalue Problems with Homogeneous Linear Constrains:
%the Lanczos Process with Inner-Outer Iterations, {\em Linear
%Algebra Appl.}, 309 (2000), pp.~289--306.

\bibitem{hochnotay} M. E. Hochstenbach and Y. Notay, Controlling
inner iterations in the Jacobi-Davidson method, {\em  SIAM J. Matrix
Analy. Appl.}, 31 (2009), pp.~460-477

\bibitem{jia09} Z. Jia, On convergence of the inexact Rayleigh
quotient iteration with the Lanczos method used for solving linear
systems, arXiv: math/0906.2239v3, submitted.


\bibitem{jiawang} Z. Jia and Z. Wang,
A convergence analysis of the inexact Rayleigh quotient iteration
and simplified Jacobi-Davidson method for the large Hermitian
matrix eigenproblem, {\em Science in China Ser.A: Mathematics}, 51
(12) (2008), pp.~2205--2216.

\bibitem{Lailin}
Y. Lai, K. Lin and W. Lin, An inexact inverse iteration for large
sparse eigenvalue problems, {\em Numer. Linear Algebra Appl.}, 4
(1997), pp.~425--437.


\bibitem{NotayRQI}
Y.~Notay, Convergence analysis of inexact {R}ayleigh quotient
iteration, {\em SIAM J. Matrix Anal. Appl.}, 24 (2003),
pp.~627--644.

%\bibitem{notayinner}
%Y.~Notay, Inner iterations in eigenvalue solvers, {\em Report
%GANMN 05-01}, Service de Metrologie Nucleaire, Universite Libre de
%Bruxelles, Belgium, 2005.

\bibitem{paige} C. C. Paige, B. N. Parlett and H. A. van der Vorst,
Approximate solutions and eigenvalue bounds from Krylov subspaces,
{\em Numer. Linear Algebra Appl.}, 2 (1995), pp.~115--134.

\bibitem{robbe} M. Robbe, M. Sadkane and A. Spence, Inexact inverse
subspace iteration with preconditioning applied to non-Hermitian
eigenvalue problems, {\em SIAM. J. Matrix Anal. Appl.}, 31 (2009),
pp.~92--113.

\bibitem{ParlettSEP} B. N. Parlett, {\em The Symmetric Eigenvalue Problem},
\newblock SIAM, Philadelphia, PA, 1998.

\bibitem{saad} Y. Saad, {\em Iterative Methods of Large Sparse
Linear Systems}, 2nd Edition, SIAM, Philadelphia, PA, 2003.

\bibitem{SimonciniRQI}
V.~Simoncini and L.~Eld$\acute{e}$n, Inexact {R}ayleigh
quotient-type methods for eigenvalue computations, {\em BIT}, 42
(2002), pp.~159--182.

\bibitem{Smit}
P.~Smit and M. H. C. Paardekooper, The effects of inexact solvers
in algorithms for symmetric eigenvalue problems, {\em Linear
Algebra Appl.}, 287 (1999), pp.~337--357.

%\bibitem{stath} A. Stathopoulos and Y. Saad,
%Restarting techniques for the (Jacobi-)Davidson eigenvalue
%methods, {\em Electr. Trans. Numer. Anal.}, 7 (1998),
%pp.~163--181.

\bibitem{stewart}
G. W. Stewart, {\em Matrix Algorithms Vol. II: Eigensystems},
\newblock SIAM, Philadelphia, PA, 2001.

\bibitem{EshofJD}
J.~van~den Eshof, The convergence of {J}acobi-{D}avidson
iterations for {H}ermitian eigenproblems, {\em Numer. Linear
Algebra Appl.}, 9 (2002), pp.~163--179.

\bibitem{vorst} H. A. van der Vorst, {\em Computational Methods for Large Eigenvalue
Problems}, In P. G. Ciarlet and J. L. Lions (eds), Handbook of
Numerical Analysis, Vol. VIII, North-Holland, Elsevier, pp.~3--179,
2002.

\bibitem{xueelman} F. Xue and H. Elman, Convergence
analysis of iterative solvers in inexact Rayleigh quotient
iteration, {\em SIAM. J. Matrix Anal. Appl.}, 31 (2009),
pp.~877--899.
\end{thebibliography}
\end{document}